\documentclass[review,hidelinks,onefignum,onetabnum]{siamart250211}
\usepackage{amsmath}
\usepackage{lipsum}
\usepackage{amsfonts}
\usepackage{graphicx}
\usepackage{subfig}
\usepackage{epstopdf}
\usepackage{algorithmic}
\usepackage{amsopn}

\ifpdf
  \DeclareGraphicsExtensions{.eps,.pdf,.png,.jpg}
\else
  \DeclareGraphicsExtensions{.eps}
\fi

\newsiamremark{remark}{Remark}
\newsiamremark{hypothesis}{Hypothesis}
\crefname{hypothesis}{Hypothesis}{Hypotheses}
\newsiamthm{claim}{Claim}
\newsiamremark{fact}{Fact}
\crefname{fact}{Fact}{Facts}
\headers{Forecasting Public Sentiments via Mean Field Games}{M. V. Klibanov, K. McGoff, and T. Truong}

\newcommand{\height}{4.5cm}

\def\func#1{\mathop{\rm #1}\nolimits}
\def\dint{\mathop{\displaystyle \int}}
\def\dsum{\mathop{\displaystyle \sum }}
\ifpdf
\hypersetup{
  pdftitle={Forecasting Public Sentiments via Mean Field Games},
  pdfauthor={M. V. Klibanov, K. McGoff, and T. Truong}
}
\fi

\begin{document}

\title{Forecasting Public Sentiments via Mean Field Games\thanks{%
Submitted to the editors DATE. 
\funding{This work was funded by the National
Science Foundation grant DMS 2436227.}}}
\author{Michael V. Klibanov \thanks{%
Department of Mathematics and Statistics, University of North Carolina at
Charlotte, Charlotte, NC 28223 (mklibanv@charlotte.edu)} \and Kevin McGoff 
\thanks{%
Department of Mathematics and Statistics, University of North Carolina at
Charlotte, Charlotte, NC 28223 (kmcgoff1@charlotte.edu)} \and Trung Truong 
\thanks{%
Department of Mathematics and Physics, Marshall University, Huntington, WV
25755 (truongt@marshall.edu)}}
\maketitle

\begin{abstract}
{\ Motivated by the goal of forecasting public sentiments, we consider a
forecasting problem in the context of the Mean Field Games theory. We
develop a numerical method, which is a version of the so-called
convexification method. We provide theoretical convergence analysis that
establishes global convergence of the method with a convergence rate. We
also conduct numerical experiments that demonstrate the accurate performance
of the convexification technique and highlight some promising features of
this approach.}
\end{abstract}




\section{Introduction}

\label{sec:1}

{The theory of Mean Field Games (MFGs) has been used to model a wide variety
of social phenomena, including public sentiment \cite%
{Banez,Bauso,Festa,Gao,Stella}. Motivated by the problem of forecasting
public sentiment, we introduce a general forecasting problem within the
framework of MFGs. The focus of this paper is on the development of a
numerical method for that forecasting problem with a rigorously justified
global convergence property, see Definition \ref{def:1.1} for our definition
of this property.}

{\ First, we present the general mathematical model, along with the
forecasting problem we wish to consider. Next, we develop a numerical method
for this problem and carry out its convergence analysis. Finally, we present
results of our numerical studies. Our theory enables us to estimate the
accuracy of the computational solution of the considered problem via the
accuracy of the prescribed initial conditions. We point out that our
forthcoming publication \cite{Kexper} emphasizes the true importance of the
technique of this paper for some societal needs. Indeed, in our forthcoming
work we successfully apply the numerical technique developed here to a large
number of samples of real public sentiment data collected during the
COVID-19 pandemic in 2020-2022.}

{The MFG theory was originated in seminal works of Lasry and Lions \cite%
{LL1,LL2} and Huang, Caines, and Malham\'{e} \cite{Huang1,Huang2}. Upon a
proper mathematical modeling, the MFG theory is capable to govern many
societal phenomena via a system of two coupled nonlinear parabolic PDEs with
two opposite directions of time. We call this the \textquotedblleft Mean
Field Games System" (MFGS). The MFG theory studies the behavior of
infinitely many rationally acting agents. This theory has numerous
applications in a variety of areas, including finance \cite{A,Trusov},
sociology \cite{Bauso}, election dynamics \cite{Chow}, etc.}

{\ Here we consider the MFGS of the following form \cite{A}:%
\begin{equation}
\left. 
\begin{array}{c}
u_{t}(x,t)+\Delta u(x,t)-r(x,t)|\nabla u(x,t)|^{2}/2+\mathop{\displaystyle
\int}\limits_{\Omega }K(x,y)m(y,t)\,dy=0, \\ 
m_{t}(x,t)-\Delta m(x,t)-\mathrm{div}(r(x,t)m(x,t)\nabla u(x,t))=0.%
\end{array}%
\right.  \label{1.1}
\end{equation}
Typically initial and terminal conditions of the form $m(x,0)=p(x)$ and $%
u(x,T)=g(x)$ are considered. We note that this system of equations presents
four main challenges: }

\begin{enumerate}
\item High nonlinearity of both PDEs.

\item Opposite directions of time in two equations (\ref{1.1}).

\item The presence of an integral operator in one of them.

\item The presence of the Laplace operator of the solution $u(x,t)$ of the
first equation in the second one.
\end{enumerate}

In this work we are interested in a forecasting problem (described precisely
in Section \ref{sec:2} below). Informally, this problem supposes that we are
given knowledge of both the density function $m(x,t)$ and the value function 
$u(x,t)$ at the initial time $t=0$, and then it asks us to predict these
functions for future times $t\in (0,T)$. However, a time marching scheme
cannot work here, see Remark 2.1 in Section 2 for a detailed explanation. In
applied situations, one may hope to obtain knowledge of the solution of the
MFGS at the initial time by using past observations of the system (including
surveys) or by numerically experimenting with various choices of initial
conditions to evaluate the response of the system to hypothetical scenarios.
In fact, we take a combination of these approaches in our forthcoming
publication \cite{Kexper}.

\begin{definition} \label{def:1.1} Given a nonlinear mathematical problem, we call a
method of its numerical solution globally convergent if there is a theorem
claiming the convergence of this method to the true solution of that problem
regardless of any \textit{a priori} knowledge of a sufficiently small neighborhood of
this solution.
\end{definition}

{Our goal here is to develop a globally convergent numerical method for our
forecasting problem. However, due to the challenges listed above, it is not
easy to construct such a method. For example, due to the high degree of
nonlinearity of the MFGS, any conventional least squares cost functional for
our problem would be non-convex, which would likely result in the phenomenon
of multiple local minima and ravines of this functional. In contrast, the
so-called \textquotedblleft convexification method" has the global
convergence property, see this section below as well as section 4 for some
details about this method. Therefore, we develop here a version of the
convexification method for our problem. Although some other versions of this
method were developed in the past for some other problems for the MFGS \cite%
{MFGbook,MFG7,MFGCAMWA}, it is impossible to automatically \textquotedblleft
project" them on the problem considered here. Indeed, any problem for which
the convexification was developed in the past has its own peculiarities,
which are necessary to address via a non-trivial convergence analysis. For
example, the key element, the so-called \textquotedblleft Carleman Weight
Function," is different for different problems, which causes significant
differences in the convergence analysis. Thus, the version of the
convexification method for the forecasting problem presented here is new. }

Opinion dynamics have been the subject of substantial interest within
various mathematical modeling frameworks; see \cite{Chamley} and \cite%
{Mossel} for surveys. The MFG theory serves as a particularly appealing
model choice for opinion dynamics, as they attempt to capture the group
dynamics of interacting agents in the large population limit. Here we
provide a brief summary of some previous modeling work in this direction. {%
Shortly after the introduction of MFGs, they were used to model certain
types of opinion dynamics; for example, \cite{Stella} provides early
modeling work in which crowd-seeking opinion dynamics are modeled with MFGs.
Later, motivated by large social networks with multiple distinct groups, 
\cite{Banez} considers a specific multi-population MFG system for opinion
dynamics and proposes a numerical scheme based on an adjoint method. With
similar motivation but a somewhat different perspective, \cite{Bauso}
considers a large population of social networks in which each network
attempts to align its state to the mean state of the other networks. In \cite%
{Festa}, the authors study a certain system of mean field equations that
they call a Forward-Forward Mean Field Game (FF-MFG), although it is not a
proper MFG. Then they model some crowd-seeking opinion dynamics in the large
population limit using their FF-MFG. With motivation again coming from
opinion dynamics on large social networks, the authors of \cite{Gao}
investigate a family of MFGs with a terminal condition designed to push the
population towards consensus, and they propose a numerical solution method
based on generative adversarial networks (GANs). We note that none of these
works provides a numerical solution scheme with rigorous convergence
guarantees. The method proposed here addresses this gap.}

The main mathematical apparatus of this paper is the method of Carleman
estimates see, e.g. books \cite{Isakov,KL} about Carleman estimates. The
tool of Carleman estimates was first introduced in the MFG theory in \cite%
{MFG1}. Since then it was applied in a number of publications to study both
analytical and numerical questions for the MFGS, including CIPs with single
measurement data for this system \cite{Chofri,MFGbook,MFG4,MFG7,MFGCAMWA}.
In addition, we refer to, e.g. \cite{Chow,Liu2,Liu1,Ren,Ren2,Yu} for some
other approaches for inverse problems for the MFGS.

In this paper we construct the above mentioned so-called convexification
method for the forecasting problem. This method was originally introduced in 
\cite{Klib97,Klib95} for two Coefficient Inverse Problems (CIPs) for
hyperbolic PDEs. The key feature of the convexification is that it avoids
the well known phenomenon of multiple local minima and ravines of
conventional least squares cost functionals for CIPs. Thus, this technique
has the global convergence property (as in Definition \ref{def:1.1}). The
convexification method works for both ill-posed problems for quasilinear
PDEs \cite[Section 5.3]{KL} and for CIPs for PDEs \cite{KL}. In particular,
we refer to \cite{MFGbook,MFG7,MFGCAMWA} for other applications of the
convexification method to the MFGS.

The convexification constructs a weighted globally strongly convex
Tikhonov-like functional for the considered problem. As mentioned above, the
key element of the convexification is the presence of a special weight
function in that functional. This is the so-called Carleman Weight Function
(CWF). The CWF is the function, which is involved as the weight function in
the Carleman estimate for the underlying PDE operator. For each version of
the convexification, the global convergence of the gradient descent method
of the minimization of that functional to the true solution of the original
problem is proven and an explicit convergence rate is presented. Thus, a
good initial guess about the solution is not required.

All functions considered below are real-valued. In section 2 we pose our
problem. In section 3 we formulate two Carleman estimates. In section 4 we
first construct the convexification functional. {Next, we provide a high
level intuitive explanation of why this functional is indeed globally
strongly convex.} Its strong convexity is proven in section 5. In section 6
we estimate the accuracy of the minimizer of that functional depending on
the level of the error in the input data. In section 7 we formulate and
prove a global convergence theorem for the gradient descent method of the
minimization of our functional. Finally, section 8 is devoted to numerical
experiments.

\section{Statement of the Problem}

\label{sec:2}

We develop our theory in the $n-$D case. However, numerical testing is
provided for the 1-D case, since the 1-D scenario makes a clear sense for
our target application of public sentiment forecasting. Let $\Omega \subset 
\mathbb{R}^{n}$ be a bounded domain, which we may interpret as the space of
possible sentiment vectors that agents may hold. We assume that the boundary 
$\partial \Omega $ of this domain is piecewise smooth. Let $x\in \Omega $
denote the position (sentiment) of an agent, and let $t\in \left( 0,T\right) 
$ denote time. Let 
\begin{equation}
\gamma \in \left( 0,1\right)  \label{2.00}
\end{equation}%
be a number. Denote%
\begin{equation}
\left. 
\begin{array}{c}
Q_{T}=\Omega \times \left( 0,T\right) ,\text{ }S_{T}=\partial \Omega \times
\left( 0,T\right) , \\ 
Q_{\gamma T}=\Omega \times \left( 0,\gamma T\right) \subset Q_{T}.%
\end{array}%
\right.  \label{2.0}
\end{equation}%
Let $u\left( x,t\right) $ and $m\left( x,t\right) $ be the value and density
functions, respectively. In terms of public sentiments, $m(x,t)$ is the
density of individuals in the population holding sentiment $x$ at time $t$,
and $u(x,t)$ can be interpreted as the expected cumulative reward of an
individual holding sentiment $x$ at time $t$. We consider the same MFGS as
the one (\ref{1.1}). We rewrite it here for the convenience of our further
derivations:%
\begin{equation}
\left. 
\begin{array}{c}
L_{1}\left( u,m\right) =u_{t}(x,t)+\Delta u(x,t){-r(x,t)(\nabla u(x,t))^{2}/2%
}+ \\ 
+\mathop{\displaystyle \int}\limits_{\Omega }K\left( x,y\right) m\left(
y,t\right) dy=0,\text{ }\left( x,t\right) \in Q_{T},%
\end{array}%
\right.  \label{2.1}
\end{equation}%
\begin{equation}
\left. 
\begin{array}{c}
L_{2}\left( u,m\right) =m_{t}(x,t)-\Delta m(x,t){-\mathop{\rm div}%
\nolimits(r(x,t)m(x,t)\nabla u(x,t))}= \\ 
=0,\text{ }\left( x,t\right) \in Q_{T}.%
\end{array}%
\right.  \label{2.2}
\end{equation}

Here, the coefficient $r(x,t)$ characterizes the reaction of the controlled
object to an action applied at the point $x$ \cite[section 5]{MFG1}. The
kernel $K(x,y)$ of the integral in (\ref{2.1}) describes an action on the
agent who occupies the state $x$ by the agent who occupies the state $y$.
This means that the integral term in (\ref{2.1}) describes the average
action on an agent who occupies the state $x$ by all other agents \cite[%
section 5]{MFG1}.

We impose zero Neumann boundary conditions, i.e. the full reflection from
the boundary%
\begin{equation}
\partial _{\nu }u\mid _{S_{T}}=\partial _{\nu }m\mid _{S_{T}}=0,  \label{2.3}
\end{equation}%
where $\nu =\nu \left( x\right) $ is the unit outward looking normal vector
at the point $\left( x,t\right) \in S_{T}.$ Conditions (\ref{2.3}) mean that
there is no flux of agents through the boundary. Denote%
\begin{equation*}
\left. H_{0}^{2}\left( Q_{T}\right) =\left\{ u\in H^{2}\left( Q_{T}\right)
:\partial _{\nu }u\mid _{S_{T}}=0\right\} .\right.
\end{equation*}%
Due to (\ref{2.3}) we assume below that functions $u,m\in H_{0}^{2}\left(
Q_{T}\right).$ In order to the forecasting problem precisely, we assume that
initial conditions for the functions $u$ and $m$ are given: 
\begin{equation}
u\left( x,0\right) =u_{0}\left( x\right) ,\text{ }m\left( x,0\right)
=m_{0}\left( x\right) ,\text{ }x\in \Omega .  \label{2.4}
\end{equation}

Recall that in the conventional case one terminal condition $u\left(
x,T\right) =u_{T}\left( x\right) $ and one initial condition $m\left(
x,0\right) =m_{0}\left( x\right) $ are given \cite{A}.{\ As to the
uniqueness for the conventional case, we refer to \cite[Theorem 2.5]{LL2}
for such a result with a certain monotonicity assumption. }However, if the
modeling goal is to forecast the system into the future, then it is natural
to assume that we only have knowledge of the system at the initial time, as
in (\ref{2.4}).

We focus in this paper on the following problem:

\textbf{Forecasting Problem}. {Find numerically functions }$u\left(
x,t\right) ,m\left( x,t\right) \in H_{0}^{2}\left( Q_{T}\right) ${\
satisfying conditions (\ref{2.1})-(\ref{2.4}) for }$\left( x,t\right) \in
\Omega \times \left( 0,t_{0}\right) ,$ where $t_{0}\in \left( 0,T\right] $
is a certain moment of time in the future. In other words, predict a future
for times $t\in \left( 0,t_{0}\right) .$

Both H\"{o}lder stability estimate and uniqueness for this problem were
proven in \cite{MFG4} and \cite[section 2.5]{MFGbook}, using the Carleman
estimates formulated in the next section.

{\ 
\begin{remark} We point out that the initial data (\ref{2.4}) do not
allow a time marching scheme for the solution of problem (\ref{2.1})-(\ref{2.4}). This is because while the problem of the solution of equation (\ref{2.2}) with the initial data $m\left( x,0\right) =m_{0}\left( x\right) $ is stable if $t>0$ increases (if the function $u$ is known), the problem of the solution of equation (\ref{2.1})
with the initial data $u\left( x,0\right) =u_{0}\left( x\right) $ is unstable if $t>0$ increases. The latter is similar to the 
case of the heat equation with the reversed time $w_{t}+\Delta w=0$\ with the initial condition $w\left( x,0\right) =w_{0}\left(
x\right) $ and with an appropriate boundary condition. It is well
known that the problem of the solution of the latter equation for $t>0$\ is unstable. We conclude, therefore, that the minimizer of the
convexification method (see below) finds a balance between stability and
instability of solutions of equations (\ref{2.1}) and (\ref{2.2}) being
supplied by conditions (\ref{2.3}) and (\ref{2.4}), also, see Remark 8.1.
\end{remark}
}

Let $M>0$ be a certain given number. We assume below that%
\begin{equation}
\left. 
\begin{array}{c}
r\in C^{1}\left( \overline{Q}_{T}\right) ,\text{ }K\in L_{\infty }\left(
\Omega \times \Omega \right) , \\ 
\left\Vert r\right\Vert _{C^{1}\left( \overline{Q}_{T}\right) }\leq M,\text{ 
}\left\Vert K\right\Vert _{L_{\infty }\left( \Omega \times \Omega \right)
}\leq M,%
\end{array}%
\right. \text{ }  \label{2.5}
\end{equation}

\section{Carleman Estimates}

\label{sec:3}

A Carleman estimate for a partial differential operator is always proven
only for the principal part of this operator since it is independent on the
lower order terms of this operator \cite[Lemma 2.1.1]{KL}. Therefore, we
formulate in this section Carleman estimates for principal parts $\partial
_{t}\pm \Delta $ of PDE operators of equations (\ref{2.1}), (\ref{2.2}). It
is important for further derivations that the CWF needs to be the same for
both these operators. Let the number $c>2$ be such that 
\begin{equation}
\frac{c^{2}}{T+c}\geq 2.  \label{4.01}
\end{equation}%
In other words, let 
\begin{equation}
c\geq 1+\sqrt{1+2T},  \label{4.02}
\end{equation}%
and let $\lambda >2$ be a parameter. Our CWF $\varphi _{\lambda }\left(
t\right) $ is:%
\begin{equation}
\varphi _{\lambda }\left( t\right) =\exp \left[ \left( T-t+c\right)
^{\lambda }\right] ,\text{ }t\in \left( 0,T\right) .  \label{4.1}
\end{equation}

\begin{theorem}[\cite{MFG2}, section 2.3.1 of \cite{MFGbook}]\label{Theorem 3.1} {%
Let }$c${\ be the number in (\ref{4.01}), (\ref{4.02}). There exists a
sufficiently large number }$\lambda _{0,1}=\lambda _{0,1}\left(
c,Q_{T}\right) \geq 1$ {depending only on listed parameters such that} 
{the following Carleman estimate holds:}%
\begin{equation*}
\dint\limits_{Q_{T}}\left( u_{t}+\Delta u\right) ^{2}\varphi _{\lambda
}^{2}dxdt\geq 
\end{equation*}%
\begin{equation}
\geq C_{1}\sqrt{\lambda }\dint\limits_{Q_{T}}\left( \nabla u\right)
^{2}\varphi _{\lambda }^{2}dxdt+C_{1}\lambda ^{2}c^{\lambda
}\dint\limits_{Q_{T}}u^{2}\varphi _{\lambda }^{2}dxdt-  \label{4.2}
\end{equation}%
\begin{equation*}
-C_{1}e^{2c^{\lambda }}\dint\limits_{\Omega }\left( \left( \nabla u\right)
^{2}+u^{2}\right) \left( x,T\right) dx-C_{1}\lambda \left( T+c\right)
^{\lambda }e^{2\left( T+c\right) ^{\lambda }}\dint\limits_{\Omega
}u^{2}\left( x,0\right) dx,\text{ }
\end{equation*}%
\begin{equation*}
\forall \lambda \geq \lambda _{0,1},\forall u\in H_{0}^{2}\left(
Q_{T}\right) ,
\end{equation*}%
{where the number }$C_{1}=C_{1}\left( c,Q_{T}\right) >0${\ depends
only on listed parameters.}
\end{theorem}

The next \cref{Theorem 3.2} is more tricky. Indeed, this theorem does not
present a conventional Carleman estimate. Rather, we call this a
\textquotedblleft quasi-Carleman estimate". This is because two test
functions $u$ and $v$ are involved here instead of the conventional case of
just a single test function.

\begin{theorem}[a quasi-Carleman estimate, \cite{MFG2}, section
2.3.2 of \cite{MFGbook}] \label{Theorem 3.2}. {Let the number }$c${\ be the same as in (\ref%
{4.01}), (\ref{4.02}). Let the function }$g\in C^{1}\left( \overline{Q}%
_{T}\right) ${.\ Then there exists a sufficiently large number }$%
\lambda _{0,2}=\lambda _{0,2}\left( c,\left\Vert g\right\Vert _{C^{1}\left( 
\overline{Q}_{T}\right) },Q_{T}\right) \geq 1$ {depending only on
listed parameters such that the following quasi-Carleman estimate holds:}%
\begin{equation*}
\dint\limits_{Q_{T}}\left( u_{t}-\Delta u+g\Delta v\right) ^{2}\varphi
_{\lambda }^{2}dxdt\geq
\end{equation*}%
\begin{equation*}
\geq \lambda c^{\lambda -1}\dint\limits_{Q_{T}}\left( \nabla u\right)
^{2}\varphi _{\lambda }^{2}dxdt+\frac{\lambda ^{2}}{4}c^{2\lambda
-2}\dint\limits_{Q_{T}}u^{2}\varphi _{\lambda }^{2}dxdt-
\end{equation*}%
\begin{equation}
-C_{2}\lambda \left( T+c\right) ^{\lambda }\dint\limits_{Q_{T}}\left( \nabla
v\right) ^{2}\varphi _{\lambda }^{2}dxdt - C_{2}\lambda \left( T+c\right) ^{\lambda }e^{2\left( T+c\right) ^{\lambda
}}\dint\limits_{\Omega }u^{2}\left( x,0\right) dx, \label{4.3}
\end{equation}%
{for all $\lambda \geq \lambda _{0,2}$ and $u,v\in H_{0}^{2}\left(Q_{T}\right)$, where the number }$C_{2}=C_{2}\left( c,\left\Vert g\right\Vert
_{C^{1}\left( \overline{Q}_{T}\right) },Q_{T}\right) >0${\ depends only
on listed parameters.}
\end{theorem}

\section{The Convexification Functional}

\label{sec:4}

\subsection{The functional}

\label{sec:4.1}

First, we define the set of pairs $\left( u,m\right) $ of functions, on
which we want to find numerically an approximate solution of problem (\ref%
{2.1})-(\ref{2.4}). We need functions 
\begin{equation}
u,m\in C^{3}\left( \overline{Q}_{T}\right) \cap H_{0}^{2}\left( Q_{T}\right)
.  \label{5.1}
\end{equation}%
Denote 
\begin{equation}
k_{n}=\left[ \frac{n+1}{2}\right] +4,  \label{5.2}
\end{equation}%
where $\left[ \left( n+1\right) /2\right] $ denotes the largest integer
which does not exceed $\left( n+1\right) /2.$ By Sobolev embedding theorem 
\begin{equation}
\left. 
\begin{array}{c}
H^{k_{n}}\left( Q_{T}\right) \subset C^{3}\left( \overline{Q}_{T}\right) ,
\\ 
\left\Vert v\right\Vert _{C^{3}\left( \overline{Q}_{T}\right) }\leq
C\left\Vert v\right\Vert _{H^{k_{n}}\left( Q_{T}\right) },\text{ }\forall
v\in H^{k_{n}}\left( Q_{T}\right) ,%
\end{array}%
\right.  \label{5.3}
\end{equation}%
where the number $C=C\left( Q_{T}\right) $ depends only on the domain $%
Q_{T}\subset \mathbb{R}^{n+1}.$ Introduce the following spaces:%
\begin{equation}
\left. 
\begin{array}{c}
H_{2}^{k_{n}}\left( Q_{T}\right) =\left\{ 
\begin{array}{c}
\left( u,m\right) \in H^{k_{n}}\left( Q_{T}\right) \times H^{k_{n}}\left(
Q_{T}\right) : \\ 
\left\Vert \left( u,m\right) \right\Vert _{H_{2}^{k_{n}}\left( Q_{T}\right)
}^{2}=\left\Vert u\right\Vert _{H^{k_{n}}\left( Q_{T}\right)
}^{2}+\left\Vert m\right\Vert _{H^{k_{n}}\left( Q_{T}\right) }^{2}<\infty%
\end{array}%
\right\} , \\ 
H_{2,0}^{k_{n}}\left( Q_{T}\right) =\left\{ 
\begin{array}{c}
\left( u,m\right) \in H_{2}^{k_{n}}\left( Q_{T}\right) : \\ 
\partial _{\nu }u\mid _{S_{T}}=\partial _{\nu }m\mid _{S_{T}}=0,%
\end{array}%
\right\} , \\ 
H_{2,0,0}^{k_{n}}\left( Q_{T}\right) =\left\{ \left( u,m\right) \in
H_{2,0}^{k_{n}}\left( Q_{T}\right) :u\left( x,0\right) =m\left( x,0\right)
=0\right\} , \\ 
H_{2,0}^{k_{n}}\left( \Omega \right) =\left\{ 
\begin{array}{c}
\left( f\left( x\right) ,g\left( x\right) \right) :\left\Vert \left(
f,g\right) \right\Vert _{H_{2}^{k_{n}}\left( \Omega \right) }^{2}= \\ 
=\left\Vert f\right\Vert _{H^{k_{n}}\left( \Omega \right) }^{2}+\left\Vert
g\right\Vert _{H^{k_{n}}\left( \Omega \right) }^{2}<\infty , \\ 
\partial _{\nu }f\mid _{\partial \Omega }=\partial _{\nu }g\mid _{\partial
\Omega }=0.%
\end{array}%
\right\} , \\ 
\text{below }\left[ ,\right] \text{ denotes the scalar product in }%
H_{2}^{k_{n}}\left( Q_{T}\right) .%
\end{array}%
\right.  \label{5.30}
\end{equation}%
Recall that the number $\gamma $ and the domain $Q_{\gamma T}$ are defined
in (\ref{2.00}) and (\ref{2.0}) respectively. In addition, we define the
space $H^{1,0}\left( Q_{\gamma T}\right) $ as:%
\begin{equation*}
H^{1,0}\left( Q_{\gamma T}\right) =\left\{ u:\left\Vert u\right\Vert
_{H^{1,0}\left( Q_{\gamma T}\right) }^{2}=\left\Vert \nabla u\right\Vert
_{L_{2}\left( Q_{\gamma T}\right) }^{2}+\left\Vert u\right\Vert
_{L_{2}\left( Q_{\gamma T}\right) }^{2}<\infty ,\text{ }\gamma \in \left(
0,1\right) \right\} .
\end{equation*}

Let $R>0$ be an arbitrary number. We assume that initial conditions in (\ref%
{2.4}) are such that 
\begin{equation}
\left( u_{0},m_{0}\right) \in H_{2,0}^{k_{n}}\left( \Omega \right)
,\left\Vert \left( u_{0},m_{0}\right) \right\Vert _{H_{2}^{k_{n}}\left(
\Omega \right) }<R.  \label{5.31}
\end{equation}%
We define the set $B\left( R\right) \subset H_{2,0}^{k_{n}}\left(
Q_{T}\right) ,$ in which we search our solution of Problem (\ref{2.1})-(\ref%
{2.4}), as:%
\begin{equation}
B\left( R\right) =\left\{ 
\begin{array}{c}
\left( u,m\right) \in H_{2,0}^{k_{n}}\left( Q_{T}\right) :\left\Vert \left(
u,m\right) \right\Vert _{H_{2}^{k_{n}}\left( Q_{T}\right) }<R, \\ 
u\left( x,0\right) =u_{0}\left( x\right) ,\text{ }m\left( x,0\right)
=m_{0}\left( x\right)%
\end{array}%
\right\} .  \label{5.4}
\end{equation}%
By (\ref{5.30}) and (\ref{5.4})%
\begin{equation}
B\left( R\right) \subset H_{2,0}^{k_{n}}\left( Q_{T}\right) .  \label{6.2}
\end{equation}%
Also, by (\ref{5.1})-(\ref{5.4})%
\begin{equation}
\left. 
\begin{array}{c}
u,m\in C^{3}\left( \overline{Q}_{T}\right) \cap H_{0}^{2}\left( Q_{T}\right)
,\text{ }\forall \left( u,m\right) \in B\left( R\right) \\ 
\left\Vert u\right\Vert _{C^{3}\left( \overline{Q}_{T}\right) }\leq CR,\text{
}\left\Vert m\right\Vert _{C^{3}\left( \overline{Q}_{T}\right) }\leq CR,%
\text{ }\forall \left( u,m\right) \in B\left( R\right) .%
\end{array}%
\right.  \label{5.5}
\end{equation}

Let $L_{1}\left( u,m\right) $ and $L_{2}\left( u,m\right) $ be two operators
in the left hand sides of of equations (\ref{2.1}), (\ref{2.2}). Denote%
\begin{equation}
q_{\lambda }\left( c,T\right) =q=\frac{1}{\lambda \left( T+c\right)
^{\lambda -1}}.  \label{5.51}
\end{equation}%
Introduce the convexification functional%
\begin{equation}
\left. 
\begin{array}{c}
J_{\lambda ,\alpha }\left( u,m\right) :\overline{B\left( R\right) }%
\rightarrow \mathbb{R}, \\ 
J_{\lambda ,\alpha }\left( u,m\right) =e^{-2ac^{\lambda }}%
\mathop{\displaystyle \int}\limits_{Q_{T}}\left[ \left( L_{1}\left(
u,m\right) \right) ^{2}+qd\left( L_{2}\left( u,m\right) \right) ^{2}\right]
\varphi _{\lambda }^{2}dxdt+ \\ 
+\alpha \left( \left\Vert u\right\Vert _{H^{k_{n}}\left( Q_{T}\right)
}^{2}+\left\Vert m\right\Vert _{H^{k_{n}}\left( Q_{T}\right) }^{2}\right) ,%
\end{array}%
\right.  \label{5.6}
\end{equation}%
where $\alpha \in \left( 0,1\right) $ is the regularization parameter and $%
a>1$ and $d>0$ are two numbers of one's choice. Both of these numbers are
independent on $\lambda $, and both should be chosen computationally. Here
the multiplier $e^{-2ac^{\lambda }}$ is taken to partially balance the
integral term with the regularization term. Indeed, by (\ref{4.1}) 
\begin{equation}
\max_{t\in \left[ 0,T\right] }\varphi _{\lambda }^{2}\left( t\right) =\exp
\left( 2\left( T+c\right) ^{\lambda }\right) .  \label{5.7}
\end{equation}%
Our goal below is to solve the following problem:

\textbf{Minimization Problem. }{Minimize the functional }$J_{\lambda ,\alpha
}\left( u,m\right) ${\ on the set }$\overline{B\left( R\right) }.$

\subsection{A high level intuitive sketch on why the functional $J_{\protect%
\lambda ,\protect\alpha }$ is strongly convex on $\overline{B\left( R\right) 
}$}

\label{sec:4.2}

{Roughly speaking, Carleman estimates of \cref{Theorem 3.1} and %
\cref{Theorem 3.2} basically ensure that the squares of weighted $%
L_{2}\left( Q_{T}\right) -$norms of principal parts of operators $L_{1}$ and 
$L_{2}$ dominate the squares of weighted $L_{2}\left( Q_{T}\right) -$norms
of lower order derivatives, where weight is the CWF $\varphi_{\lambda
}\left( t\right) $ in (\ref{4.1}), and the parameter $\lambda \geq 1$ is
sufficiently large. However, since those principal parts are linear
operators, then squares of their weighted $L_{2}\left( Q_{T}\right) -$norms
are quadratic functionals. Hence, again very roughly (i.e. ignoring initial
and terminal conditions), those quadratic functionals are strongly convex on
a corresponding Hilbert space. On the other hand, the above mentioned
domination ensures that the presence of terms with lower order derivatives
in operators $L_{1}$ and $L_{2}$ does not affect the global convexity
property, at least when working with the set $\overline{B\left( R\right) }.$
Since $R>0$ is an arbitrary number, then this is global strong convexity. }

\section{The Global Strong Convexity Property}

\label{sec:5}

In this section we prove the key analytical result of our paper: the strong
convexity property of the functional $J_{\lambda ,\alpha }\left( u,m\right) $
on the set $\overline{B\left( R\right) }.$ Since smallness assumptions are
not imposed on R, then we call this \textquotedblleft global strong
convexity property". Recall that the number $M>0$ was introduced in (\ref%
{2.5}) and that in (\ref{5.6}) we must have $a>1.$

\begin{theorem}[the central result]\label{Theorem 5.1} {Assume that conditions (\ref%
{2.00}), (\ref{2.5}), (\ref{4.01}), (\ref{4.1}), (\ref{5.2}), (\ref{5.51})
and (\ref{5.6}) hold. Then:}

{1. At each point }$\left( u,m\right) \in \overline{B\left( R\right) }$%
{\ there exists Fr\'{e}chet derivative }$J_{\lambda ,\alpha }^{\prime
}\left( u,m\right) \in H_{2,0,0}^{k_{n}}\left( Q_{T}\right) ${\ of the
functional }$J_{\lambda ,\alpha }\left( u,m\right) .${\ Furthermore,
this derivative is Lipschitz continuous on }$\overline{B\left( R\right) }$%
{, i.e. there exists a number }$D=D\left( R,\lambda ,c,a,d,M,\alpha
,Q_{T}\right) ${\ depending only on listed parameters such that} 
\begin{equation}
\left. 
\begin{array}{c}
\left\Vert J_{\lambda ,\alpha }^{\prime }\left( u_{1},m_{1}\right)
-J_{\lambda ,\alpha }^{\prime }\left( u_{2},m_{2}\right) \right\Vert
_{H_{2}^{k_{n}}\left( Q_{T}\right) }\leq D\left\Vert \left(
u_{1},m_{1}\right) -\left( u_{2},m_{2}\right) \right\Vert
_{H_{2}^{k_{n}}\left( Q_{T}\right) }, \\ 
\forall \left( u_{1},m_{1}\right) ,\left( u_{2},m_{2}\right) \in \overline{%
B\left( R\right) }.%
\end{array}%
\right.  \label{6.3}
\end{equation}

{2. There exists a sufficiently large number }$\lambda _{1}=\lambda
_{1}\left( M,R,c,a,d,Q_{T}\right) \geq 1${\ depending only on listed
parameters such that if }$\lambda \geq \lambda _{1}${\ and if the
regularization parameter}%
\begin{equation}
\alpha \in \left[ 2e^{-\left( a-1\right) c^{\lambda }},1\right) ,
\label{6.4}
\end{equation}%
{then the functional }$J_{\lambda ,\alpha }\left( u,m\right) ${\
is strongly convex on the set }$\overline{B\left( R\right) }.${\ More
precisely, the} {following strong convexity estimate is valid:}%
\begin{equation}
\left. 
\begin{array}{c}
J_{\lambda ,\alpha }\left( u_{2},m_{2}\right) -J_{\lambda ,\alpha }\left(
u_{1},m_{1}\right) -\left[ J_{\lambda ,\alpha }^{\prime }\left(
u_{1},m_{1}\right) ,\left( u_{2}-u_{1},m_{2}-m_{1}\right) \right] \geq \\ 
\geq C_{3}e^{-2ac^{\lambda }}\left( c/\left( T+c\right) \right) ^{\lambda
}\exp \left[ 2\left( T\left( 1-\gamma \right) +c\right) ^{\lambda }\right]
\times \\ 
\times \left( \left\Vert u_{2}-u_{1}\right\Vert _{H^{1,0}\left( Q_{\gamma
T}\right) }^{2}+\left\Vert m_{2}-m_{1}\right\Vert _{H^{1,0}\left( Q_{\gamma
T}\right) }^{2}\right) + \\ 
+\left( \alpha /2\right) \left\Vert \left( u_{2}-u_{1},m_{2}-m_{1}\right)
\right\Vert _{H_{2}^{k_{n}}\left( Q_{T}\right) }^{2}, \\ 
\forall \left( u_{1},m_{1}\right) ,\left( u_{2},m_{2}\right) \in \overline{%
B\left( R\right) },\text{ }\forall \lambda \geq \lambda _{1},%
\end{array}%
\right.  \label{6.5}
\end{equation}%
{where the number }$C_{3}=C_{3}\left( M,R,c,d,Q_{T}\right) >0${\
depends only on listed parameters.}

3. Let $\lambda \geq \lambda_{1}$ and let (\ref{6.4}) holds. Then there exists unique minimizer \\ $\left( u_{\min ,\lambda},m_{\min,\lambda}\right) \in \overline{B\left( R\right)}$ of the functional $J_{\lambda ,\alpha }\left( u,m\right) ${\ on the set }$\overline{B\left( R\right)}$  and the following estimate holds
\begin{equation}
\left[ J_{\lambda ,\alpha }^{\prime }\left( u_{\min ,\lambda },m_{\min
,\lambda }\right) ,\left( u_{\min ,\lambda }-u,m_{\min ,\lambda }-m\right) %
\right] \leq 0,\text{ }\forall \left( u,m\right) \in \overline{B\left(
R\right) }.  \label{6.6}
\end{equation}
\end{theorem}

\begin{proof} Below in this paper $C_{3}=C_{3}\left(
M,R,c,a,d,Q_{T}\right) >0$ denotes different numbers depending only on
listed parameters. Let $\left( u_{1},m_{1}\right) ,\left( u_{2},m_{2}\right)
\in \overline{B\left( R\right) }$ be two arbitrary points. Denote $\left(
h_{1},h_{2}\right) =\left( u_{2}-u_{1},m_{2}-m_{1}\right) .$ Then 
\begin{equation}
\left( u_{2},m_{2}\right) =\left( u_{1}+h_{1},m_{1}+h_{2}\right) .
\label{6.7}
\end{equation}%
In addition, by (\ref{5.4}), (\ref{6.2}) and triangle inequality 
\begin{equation}
\left( h_{1},h_{2}\right) \in \overline{B_{0}\left( 2R\right) }=\left\{
\left( u,m\right) \in H_{2,0,0}^{k_{n}}\left( Q_{T}\right) :\left\Vert
\left( u,m\right) \right\Vert _{H_{2}^{k_{n}}\left( Q_{T}\right) }\leq
2R\right\} .  \label{6.8}
\end{equation}%
It is convenient to represent the functional {\ }$J_{\lambda ,\alpha
}\left( u,m\right) $ as the sum of three functionals and evaluate each of
them separately, i.e. 
\begin{equation}
\left. 
\begin{array}{c}
J_{1,\lambda ,\alpha }\left( u,m\right) =e^{-2ac^{\lambda
}}\dint\limits_{Q_{T}}\left( L_{1}\left( u,m\right) \right) ^{2}\varphi
_{\lambda }^{2}dxdt, \\ 
J_{2,\lambda ,\alpha }\left( u,m\right) =e^{-2ac^{\lambda
}}qd\dint\limits_{Q_{T}}\left( L_{2}\left( u,m\right) \right) ^{2}\varphi
_{\lambda }^{2}dxdt, \\ 
J_{3,\lambda ,\alpha }\left( u,m\right) =\alpha \left( \left\Vert
u\right\Vert _{H^{k_{n}}\left( Q_{T}\right) }^{2}+\left\Vert m\right\Vert
_{H^{k_{n}}\left( Q_{T}\right) }^{2}\right) , \\ 
J_{\lambda ,\alpha }\left( u,m\right) =J_{1,\lambda ,\alpha }\left(
u,m\right) +J_{2,\lambda ,\alpha }\left( u,m\right) +J_{3,\lambda ,\alpha
}\left( u,m\right) .%
\end{array}%
\right.  \label{6.9}
\end{equation}

\textbf{Step 1. }Evaluate the functional $J_{1,\lambda ,\alpha }\left(
u,m\right) $ in (\ref{6.9}).

By (\ref{2.1}) and (\ref{6.7}) 
\begin{equation*}
\left. 
\begin{array}{c}
L_{1}\left( u_{2},m_{2}\right) =L_{1}\left( u_{1}+h_{1},m_{1}+h_{2}\right) =
\\ 
=\left( u_{1t}+\Delta u_{1}-r\left( \nabla u_{1}\right)
^{2}/2+\dint\limits_{\Omega }K\left( x,y\right) m_{1}\left( y,t\right)
dy\right) + \\ 
+\left( h_{1t}+\Delta h_{1}-r\nabla u_{1}\nabla h_{1}+\dint\limits_{\Omega
}K\left( x,y\right) h_{2}\left( y,t\right) dy\right) -r\left( \nabla
h_{1}\right) ^{2}/2.%
\end{array}%
\right.
\end{equation*}%
Hence,%
\begin{equation}
\left. 
\begin{array}{c}
e^{-2ac^{\lambda }}\left( L_{1}\left( u_{1}+h_{1},m_{1}+h_{2}\right) \right)
^{2}-e^{-3c^{\lambda }}\left( L_{1}\left( u_{1},m_{1}\right) \right) ^{2}=
\\ 
=2e^{-2ac^{\lambda }}L_{1}\left( u_{1},m_{1}\right) \left( h_{1t}+\Delta
h_{1}-r\nabla u_{1}\nabla h_{1}+\dint\limits_{\Omega }K\left( x,y\right)
h_{2}\left( y,t\right) dy\right) - \\ 
-e^{-2ac^{\lambda }}L_{1}\left( u_{1},m_{1}\right) \cdot r\left( \nabla
h_{1}\right) ^{2}+ \\ 
+e^{-2ac^{\lambda }}\left[ \left( h_{1t}+\Delta h_{1}-r\nabla u_{1}\nabla
h_{1}+\dint\limits_{\Omega }K\left( x,y\right) h_{2}\left( y,t\right)
dy\right) -r\left( \nabla h_{1}\right) ^{2}/2\right] ^{2}= \\ 
=A_{1,\text{lin}}\left( h_{1},h_{2}\right) +A_{1,\text{nonlin}}\left(
h_{1},h_{2}\right) ,%
\end{array}%
\right.  \label{6.10}
\end{equation}%
where $A_{1,\text{lin}}\left( h_{1},h_{2}\right) $ and $A_{1,\text{nonlin}%
}\left( h_{1},h_{2}\right) $ are linear and nonlinear operators respectively
with respect to $\left( h_{1},h_{2}\right) ,$%
\begin{equation}
\left. 
\begin{array}{c}
A_{1,\text{lin}}\left( h_{1},h_{2}\right) = \\ 
=2e^{-2ac^{\lambda }}L_{1}\left( u_{1},m_{1}\right) \left( h_{1t}+\Delta
h_{1}-r\nabla u_{1}\nabla h_{1}+\dint\limits_{\Omega }K\left( x,y\right)
h_{2}\left( y,t\right) dy\right) ,%
\end{array}%
\right.  \label{6.11}
\end{equation}%
\begin{equation}
\left. 
\begin{array}{c}
A_{1,\text{nonlin}}\left( h_{1},h_{2}\right) =-e^{-2ac^{\lambda
}}L_{1}\left( u_{1},m_{1}\right) \cdot r\left( \nabla h_{1}\right) ^{2}+ \\ 
+e^{-2ac^{\lambda }}\left[ \left( h_{1t}+\Delta h_{1}-r\nabla u_{1}\nabla
h_{1}+\dint\limits_{\Omega }K\left( x,y\right) h_{2}\left( y,t\right)
dy\right) -r\left( \nabla h_{1}\right) ^{2}/2\right] ^{2}.%
\end{array}%
\right.  \label{6.12}
\end{equation}%
Using the first line of (\ref{6.9}) as well as (\ref{6.10})-(\ref{6.12}), we
obtain 
\begin{equation}
\left. 
\begin{array}{c}
J_{1,\lambda ,\alpha }\left( u_{2},m_{2}\right) -J_{1,\lambda ,\alpha
}\left( u_{1},m_{1}\right) =J_{1,\lambda ,\alpha }\left(
u_{1}+h_{1},m_{1}+h_{2}\right) -J_{1,\lambda ,\alpha }\left(
u_{1},m_{1}\right) = \\ 
=\dint\limits_{Q_{T}}A_{1,\text{lin}}\left( h_{1},h_{2}\right) \varphi
_{\lambda }^{2}dxdt+\dint\limits_{Q_{T}}A_{1,\text{nonlin}}\left(
h_{1},h_{2}\right) \varphi _{\lambda }^{2}dxdt.%
\end{array}%
\right.  \label{6.13}
\end{equation}

\textbf{Step 2. }Evaluate the functional $J_{2,\lambda ,\alpha }\left(
u,m\right) $ in (\ref{6.9}).

By (\ref{2.2}), (\ref{6.7}) and (\ref{6.9})%
\begin{equation*}
\left. 
\begin{array}{c}
L_{2}\left( u_{2},m_{2}\right) =L_{2}\left( u_{1}+h_{1},m_{1}+h_{2}\right)
=L_{2}\left( u_{1},m_{1}\right) + \\ 
+\left[ h_{2t}-\Delta h_{2}-\func{div}\left( rm_{1}\nabla h_{1}\right) -%
\func{div}\left( rm_{1}h_{2}\right) -\func{div}\left( r\nabla
u_{1}h_{2}\right) \right] - \\ 
-\func{div}\left( rh_{2}\nabla h_{1}\right) .%
\end{array}%
\right.
\end{equation*}%
Hence,%
\begin{equation}
\left. 
\begin{array}{c}
e^{-2ac^{\lambda }}qd\left( L_{2}\left( u_{1}+h_{1},m_{1}+h_{2}\right)
\right) ^{2}-e^{-3c^{\lambda }}qd\left( L_{2}\left( u_{1},m_{1}\right)
\right) ^{2}= \\ 
=2e^{-2ac^{\lambda }}qdL_{2}\left( u_{1},m_{1}\right) \times \\ 
\times \left[ h_{2t}-\Delta h_{2}-\func{div}\left( rm_{1}\nabla h_{1}\right)
-\func{div}\left( rm_{1}h_{2}\right) -\func{div}\left( r\nabla
u_{1}h_{2}\right) \right] - \\ 
-2e^{-2ac^{\lambda }}qdL_{2}\left( u_{1},m_{1}\right) \func{div}\left(
rh_{2}\nabla h_{1}\right) + \\ 
+e^{-2ac^{\lambda }}qd\left[ 
\begin{array}{c}
h_{2t}-\Delta h_{2}-\func{div}\left( rm_{1}\nabla h_{1}\right) -\func{div}%
\left( rm_{1}h_{2}\right) - \\ 
-\func{div}\left( r\nabla u_{1}h_{2}\right) -\func{div}\left( rh_{2}\nabla
h_{1}\right)%
\end{array}%
\right] ^{2}= \\ 
=A_{2,\text{lin}}\left( h_{1},h_{2}\right) +A_{2,\text{nonlin}}\left(
h_{1},h_{2}\right) ,%
\end{array}%
\right.  \label{6.14}
\end{equation}%
where similarly with the above $A_{2,\text{lin}}\left( h_{1},h_{2}\right) $
and $A_{2,\text{nonlin}}\left( h_{1},h_{2}\right) $ are linear and nonlinear
operators respectively with respect to $\left( h_{1},h_{2}\right) ,$%
\begin{equation}
\left. 
\begin{array}{c}
A_{2,\text{lin}}\left( h_{1},h_{2}\right) =2e^{-2ac^{\lambda }}qdL_{2}\left(
u_{1},m_{1}\right) \times \\ 
\times \left[ h_{2t}-\Delta h_{2}-\func{div}\left( rm_{1}\nabla h_{1}\right)
-\func{div}\left( rm_{1}h_{2}\right) -\func{div}\left( r\nabla
u_{1}h_{2}\right) \right] ,%
\end{array}%
\right.  \label{6.15}
\end{equation}%
\begin{equation}
\left. 
\begin{array}{c}
A_{2,\text{nonlin}}\left( h_{1},h_{2}\right) =-2e^{-2ac^{\lambda
}}qdL_{2}\left( u_{1},m_{1}\right) \func{div}\left( rh_{2}\nabla
h_{1}\right) + \\ 
+e^{-2ac^{\lambda }}qd\left[ 
\begin{array}{c}
h_{2t}-\Delta h_{2}-\func{div}\left( rm_{1}\nabla h_{1}\right) -\func{div}%
\left( rm_{1}h_{2}\right) - \\ 
-\func{div}\left( r\nabla u_{1}h_{2}\right) -\func{div}\left( rh_{2}\nabla
h_{1}\right)%
\end{array}%
\right] ^{2}.%
\end{array}%
\right.  \label{6.16}
\end{equation}%
Hence,%
\begin{equation}
\left. 
\begin{array}{c}
J_{2,\lambda ,\alpha }\left( u_{2},m_{2}\right) -J_{2,\lambda ,\alpha
}\left( u_{1},m_{1}\right) = \\ 
=J_{2,\lambda ,\alpha }\left( u_{1}+h_{1},m_{1}+h_{2}\right) -J_{2,\lambda
,\alpha }\left( u_{1},m_{1}\right) = \\ 
=e^{-2ac^{\lambda }}qd\dint\limits_{Q_{T}}A_{2,\text{lin}}\left(
h_{1},h_{2}\right) \varphi _{\lambda }^{2}dxdt+e^{-2ac^{\lambda
}}qd\dint\limits_{Q_{T}}A_{2,\text{nonlin}}\left( h_{1},h_{2}\right) \varphi
_{\lambda }^{2}dxdt.%
\end{array}%
\right.  \label{6.17}
\end{equation}

\textbf{Step 3. }Evaluate the functional $J_{3,\lambda ,\alpha }\left(
u,m\right) $ in (\ref{6.9}).

We have: 
\begin{equation}
\left. 
\begin{array}{c}
J_{3,\lambda ,\alpha }\left( u_{2},m_{2}\right) -J_{3,\lambda ,\alpha
}\left( u_{1},m_{1}\right) = \\ 
=2\alpha \left[ \left( u_{1},m_{1}\right) ,\left( h_{1},h_{2}\right) \right]
+\alpha \left[ \left\Vert h_{1}\right\Vert _{H^{k_{n}}\left( Q_{T}\right)
}^{2}+\left\Vert h_{2}\right\Vert _{H^{k_{n}}\left( Q_{T}\right) }^{2}\right]
= \\ 
=A_{3,\text{lin}}\left( h_{1},h_{2}\right) +A_{3,\text{nonlin}}\left(
h_{1},h_{2}\right) , \\ 
A_{3,\text{lin}}\left( h_{1},h_{2}\right) =2\alpha \left[ \left(
u_{1},m_{1}\right) ,\left( h_{1},h_{2}\right) \right] , \\ 
A_{3,\text{nonlin}}\left( h_{1},h_{2}\right) =\alpha \left[ \left\Vert
h_{1}\right\Vert _{H^{k_{n}}\left( Q_{T}\right) }^{2}+\left\Vert
h_{2}\right\Vert _{H^{k_{n}}\left( Q_{T}\right) }^{2}\right] .%
\end{array}%
\right.  \label{6.18}
\end{equation}

Thus, it follows from (\ref{6.9})-(\ref{6.18}) that%
\begin{equation}
\left. 
\begin{array}{c}
J_{\lambda ,\alpha }\left( u_{1}+h_{1},m_{1}+h_{2}\right) -J_{\lambda
,\alpha }\left( u_{1},m_{1}\right) = \\ 
=\dsum\limits_{i=1}^{3}A_{i,\text{lin}}\left( h_{1},h_{2}\right)
+\dsum\limits_{i=1}^{3}A_{i,\text{nonlin}}\left( h_{1},h_{2}\right) .%
\end{array}%
\right.  \label{6.19}
\end{equation}%
Consider the functional $A_{\text{lin}}\left( h_{1},h_{2}\right) ,$%
\begin{equation}
A_{\text{lin}}\left( h_{1},h_{2}\right) =\dsum\limits_{i=1}^{3}A_{i,\text{lin%
}}\left( h_{1},h_{2}\right) .  \label{6.20}
\end{equation}%
Since by (\ref{6.8}) the pair $\left( h_{1},h_{2}\right) \in
H_{2,0,0}^{k_{n}}\left( Q_{T}\right) ,$ then by (\ref{6.11}), (\ref{6.15})
and (\ref{6.18}) $A_{\text{lin}}\left( h_{1},h_{2}\right) $ is a linear
bounded functional such that 
\begin{equation}
A_{\text{lin}}\left( h_{1},h_{2}\right) :H_{2,0,0}^{k_{n}}\left(
Q_{T}\right) \rightarrow \mathbb{R}.  \label{6.21}
\end{equation}%
Next, it follows from Riesz theorem, the last line of (\ref{5.30}) and (\ref%
{6.21}) that there exists unique point $\overline{A}_{\text{lin}}\in
H_{2,0,0}^{k_{n}}\left( Q_{T}\right) $ such that 
\begin{equation}
A_{\text{lin}}\left( h_{1},h_{2}\right) =\left[ \overline{A}_{\text{lin}%
},\left( h_{1},h_{2}\right) \right] ,\text{ }\forall \left(
h_{1},h_{2}\right) \in H_{2,0,0}^{k_{n}}\left( Q_{T}\right) .  \label{6.22}
\end{equation}%
Hence, using (\ref{5.5}), (\ref{6.12}), (\ref{6.16}), the third line of (\ref%
{6.18}), (\ref{6.19}) and (\ref{6.22}), we obtain%
\begin{equation*}
\lim_{\left\Vert \left( h_{1},h_{2}\right) \right\Vert _{H_{2}^{k_{n}}\left(
Q_{T}\right) }\rightarrow 0}\frac{J_{\lambda ,\alpha }\left(
u_{1}+h_{1},m_{1}+h_{2}\right) -J_{\lambda ,\alpha }\left(
u_{1},m_{1}\right) -\left[ \overline{A}_{\text{lin}},\left(
h_{1},h_{2}\right) \right] }{\left\Vert \left( h_{1},h_{2}\right)
\right\Vert _{H_{2}^{k_{n}}\left( Q_{T}\right) }}=0.
\end{equation*}%
Hence, $\overline{A}_{\text{lin}}$ is the Fr\'{e}chet derivative of the
functional $J_{\lambda ,\alpha }\left( u,m\right) $ at the point $\left(
u_{1},m_{1}\right) ,$ i.e.%
\begin{equation}
J_{\lambda ,\alpha }^{\prime }\left( u_{1},m_{1}\right) \left(
h_{1},h_{2}\right) =\left[ \overline{A}_{\text{lin}},\left(
h_{1},h_{2}\right) \right] ,\text{ }\forall \left( h_{1},h_{2}\right) \in
H_{2,0,0}^{k_{n}}\left( Q_{T}\right) .  \label{6.220}
\end{equation}%
We omit the proof of estimate (\ref{6.3}) since it is proven similarly with
the proof of Theorem 3.1 of \cite{Bak}.

\textbf{Step 4.} We now prove the strong convexity estimate (\ref{6.5}).

By (\ref{5.3}), (\ref{6.8}), (\ref{6.12}), (\ref{6.16}), (\ref{6.18}) and (%
\ref{6.19}) 
\begin{equation}
\left. 
\begin{array}{c}
\dsum\limits_{i=1}^{2}A_{i,\text{nonlin}}\left( h_{1},h_{2}\right) \geq \\ 
\geq C_{3}e^{-2ac^{\lambda }}\left( h_{1t}+\Delta h_{1}\right)
^{2}+C_{3}e^{-2ac^{\lambda }}q\left( h_{2t}-\Delta h_{2}-rm_{1}\Delta
h_{1}\right) ^{2}- \\ 
-C_{3}e^{-2ac^{\lambda }}\left( \left( \nabla h_{1}\right)
^{2}+h_{1}^{2}+q\left( \nabla h_{2}\right)
^{2}+qh_{2}^{2}+\dint\limits_{\Omega }h_{2}^{2}\left( y,t\right) dy\right) .%
\end{array}%
\right.  \label{6.23}
\end{equation}%
Let $\lambda _{0,1}\geq 1$ be the number of \cref{Theorem 3.1}. Due to the term $%
rm_{1}$ in the second line of (\ref{6.23}), we now take in \cref{Theorem 3.2} $%
\lambda _{0,2}=\lambda _{0,2}\left( \left\Vert r\right\Vert _{C^{1}\left( 
\overline{Q}_{T}\right) },R,c,Q_{T}\right) \geq 1.$ Next, we denote 
\begin{equation*}
\lambda _{0}=\lambda _{0}\left( \left\Vert r\right\Vert _{C^{1}\left( 
\overline{Q}_{T}\right) },R,c,Q_{T}\right) =\max \left( \lambda
_{0,1},\lambda _{0,2}\right)
\end{equation*}%
Multiply inequality (\ref{6.23})\ by the function $\varphi _{\lambda
}^{2}\left( t\right) $ defined in (\ref{4.1}) and integrate it then over $%
Q_{T}.$ Then apply Carleman estimate (\ref{4.2}) of \cref{Theorem 3.1} to the term $%
\left( h_{1t}+\Delta h_{1}\right) ^{2}$ and the quasi-Carleman estimate (\ref%
{4.3}) of \cref{Theorem 3.2} to the term $q\left( h_{2t}-\Delta h_{2}-rm_{1}\Delta
h_{1}\right) ^{2}.$ Taking into account (\ref{4.01}), (\ref{4.02}), (\ref%
{5.51}) and the last line of (\ref{6.8}), we obtain 
\begin{equation*}
\dint\limits_{Q_{T}}\dsum\limits_{i=1}^{2}A_{i,\text{nonlin}}\left(
h_{1},h_{2}\right) \varphi _{\lambda }^{2}dxdt\geq
\end{equation*}%
\begin{equation*}
\geq C_{3}e^{-2ac^{\lambda }}\sqrt{\lambda }\dint\limits_{Q_{T}}\left(
\nabla h_{1}\right) ^{2}\varphi _{\lambda }^{2}dxdt+C_{3}e^{-2ac^{\lambda
}}\lambda ^{2}c^{\lambda }\dint\limits_{Q_{T}}h_{1}^{2}\varphi _{\lambda
}^{2}dxdt+
\end{equation*}%
\begin{equation}
+C_{3}e^{-2ac^{\lambda }}\left( \frac{c}{T+c}\right) ^{\lambda
}\dint\limits_{Q_{T}}\left( \nabla h_{2}\right) ^{2}\varphi _{\lambda
}^{2}dxdt+C_{3}e^{-2ac^{\lambda }}\cdot 2^{\lambda
}\dint\limits_{Q_{T}}h_{2}^{2}\varphi _{\lambda }^{2}dxdt-  \label{6.24}
\end{equation}%
\begin{equation*}
-C_{3}e^{-2ac^{\lambda }}\dint\limits_{Q_{T}}\left( \left( \nabla
h_{1}\right) ^{2}+h_{1}^{2}+q\left( \nabla h_{2}\right)
^{2}+qh_{2}^{2}+\dint\limits_{\Omega }h_{2}^{2}\left( y,t\right) dy\right)
\varphi _{\lambda }^{2}dxdt-
\end{equation*}%
\begin{equation*}
-C_{3}e^{-2\left( a-1\right) c^{\lambda }}\dint\limits_{\Omega }\left[
\left( \nabla h_{1}\right) ^{2}+h_{1}^{2}\right] \left( x,T\right) dx,\text{ 
}\forall \lambda \geq \lambda _{0}.
\end{equation*}%
By (\ref{4.02}) and (\ref{5.51}) there exists a sufficiently large number $%
\lambda _{1,1}=\lambda _{1,1}\left( c,T\right) \geq \lambda _{0}$ depending
only on $c$ and $T$ such that 
\begin{equation}
\left( \frac{c}{T+c}\right) ^{\lambda }-q=\frac{1}{\left( T+c\right)
^{\lambda }}\left( c^{\lambda }-\frac{1}{\lambda }\right) \geq \frac{%
c^{\lambda }}{2\left( T+c\right) ^{\lambda }},\text{ }\forall \lambda \geq
\lambda _{1,1}.  \label{6.26}
\end{equation}%
Then (\ref{6.24}) and (\ref{6.26}) imply: 
\begin{equation*}
C_{3}e^{-2ac^{\lambda }}\left( \frac{c}{T+c}\right) ^{\lambda
}\dint\limits_{Q_{T}}\left( \nabla h_{2}\right) ^{2}\varphi _{\lambda
}^{2}dxdt-C_{3}e^{-2ac^{\lambda }}q\dint\limits_{Q_{T}}\left( \nabla
h_{2}\right) ^{2}\varphi _{\lambda }^{2}dxdt\geq
\end{equation*}%
\begin{equation}
\geq C_{3}e^{-2ac^{\lambda }}\left( \frac{c}{T+c}\right) ^{\lambda
}\dint\limits_{Q_{T}}\left( \nabla h_{2}\right) ^{2}\varphi _{\lambda
}^{2}dxdt,\text{ }\forall \lambda \geq \lambda _{1,1}.  \label{6.27}
\end{equation}%
Hence, using (\ref{6.24}) and (\ref{6.27}), we obtain that there exists a
sufficiently large number $\lambda _{1}=\lambda _{1}\left(
M,R,c,a,d,Q_{T}\right) \geq \max \left( \lambda _{1,1}\left( c,T\right)
,1\right) $ depending only on listed parameters such that 
\begin{equation*}
\dint\limits_{Q_{T}}\dsum\limits_{i=1}^{3}A_{i,\text{nonlin}}\left(
h_{1},h_{2}\right) \varphi _{\lambda }^{2}dxdt\geq
\end{equation*}%
\begin{equation*}
\geq C_{3}e^{-2ac^{\lambda }}\sqrt{\lambda }\dint\limits_{Q_{T}}\left(
\nabla h_{1}\right) ^{2}\varphi _{\lambda }^{2}dxdt+C_{3}e^{-2ac^{\lambda
}}\lambda ^{2}c^{\lambda }\dint\limits_{Q_{T}}h_{1}^{2}\varphi _{\lambda
}^{2}dxdt+
\end{equation*}%
\begin{equation}
+C_{3}e^{-2ac^{\lambda }}\left( \frac{c}{T+c}\right) ^{\lambda
}\dint\limits_{Q_{T}}\left( \nabla h_{2}\right) ^{2}\varphi _{\lambda
}^{2}dxdt+C_{3}e^{-2ac^{\lambda }}\cdot 2^{\lambda
}\dint\limits_{Q_{T}}h_{2}^{2}\varphi _{\lambda }^{2}dxdt-  \label{6.28}
\end{equation}%
\begin{equation*}
-C_{3}e^{-2\left( a-1\right) c^{\lambda }}\dint\limits_{\Omega }\left[
\left( \nabla h_{1}\right) ^{2}+h_{1}^{2}\right] \left( x,T\right) dx
+\alpha \left[ \left\Vert h_{1}\right\Vert _{H^{k_{n}}\left( Q_{T}\right)
}^{2}+\left\Vert h_{2}\right\Vert _{H^{k_{n}}\left( Q_{T}\right) }^{2}\right]
,
\end{equation*}%
for all $\lambda \geq \lambda _{1}$. By (\ref{6.4}), $\alpha /2>e^{-2\left( a-1\right) c^{\lambda }},$ $\forall
\lambda \geq \lambda _{1}.$ Hence, it follows from (\ref{6.7}), (\ref{6.19}%
)-(\ref{6.220}) and (\ref{6.28}) that%
\begin{equation*}
J_{\lambda ,\alpha }\left( u_{1}+h_{1},m_{1}+h_{2}\right) -J_{\lambda
,\alpha }\left( u_{1},m_{1}\right) -J_{\lambda ,\alpha }^{\prime }\left(
u_{1},m_{1}\right) \left( h_{1},h_{2}\right) \geq
\end{equation*}%
\begin{equation*}
\geq C_{3}e^{-2ac^{\lambda }}\left( \frac{c}{T+c}\right) ^{\lambda }\exp %
\left[ 2\left( T\left( 1-\gamma \right) +c\right) ^{\lambda }\right] \left(
\left\Vert h_{1}\right\Vert _{H^{1,0}\left( Q_{\gamma T}\right)
}^{2}+\left\Vert h_{2}\right\Vert _{H^{1,0}\left( Q_{\gamma T}\right)
}^{2}\right) +
\end{equation*}%
\begin{equation*}
+\frac{\alpha }{2}\left[ \left\Vert h_{1}\right\Vert _{H^{k_{n}}\left(
Q_{T}\right) }^{2}+\left\Vert h_{2}\right\Vert _{H^{k_{n}}\left(
Q_{T}\right) }^{2}\right] ,\text{ }
\end{equation*}%
\begin{equation*}
\forall \lambda \geq \lambda _{1},\text{ }\forall \left( u_{1},m_{1}\right)
,\left( u_{2},m_{2}\right) =\left( u_{1}+h_{1},m_{1}+h_{2}\right) \in 
\overline{B\left( R\right) },
\end{equation*}%
which proves the strong convexity inequality (\ref{6.5}). As soon as (\ref%
{6.5}) is established, the existence and uniqueness of the minimizer $\left(
u_{\min ,\lambda },m_{\min ,\lambda }\right) \in \overline{B\left( R\right) }
${\ }of the functional $J_{\lambda ,\alpha }\left( u,m\right) $\ on the
set $\overline{B\left( R\right) }$ as well as inequality (\ref{6.6}) follow
immediately from a combination of Lemma 2.1 with Theorem 2.1 of \cite{Bak}.
\end{proof}

\section{The Accuracy of the Minimizer}

\label{sec:6}

In applications, input data are always given with errors. Therefore, it is
important to figure out how the error in the initial conditions (\ref{2.4})
affects the minimizer $\left( u_{\min ,\lambda },m_{\min ,\lambda }\right) $
of the functional $J_{\lambda ,\alpha }\left( u,m\right) .$ Recall that the
existence and uniqueness of this minimizer on the set $\overline{B\left(
R\right) }$ were established in \cref{Theorem 5.1}. In this section we
estimate the accuracy of the minimizer of a functional, which is closely
related to $J_{\lambda ,\alpha }\left( u,m\right) .$ The reason of the
replacement of $J_{\lambda ,\alpha }\left( u,m\right) $ with that new
functional is that it follows from the second line of (\ref{5.4}) and (\ref%
{6.6}) that we can compare only pairs of functions $\left( u,m\right) ,$
which have the same initial conditions $\left( u\left( x,0\right) ,m\left(
x,0\right) \right) .$

To get the desired estimate, we first follow one of the main concepts of the
theory of Ill-Posed Problems \cite{T} via assuming that there exists an
ideal solution $\left( u^{\ast },m^{\ast }\right) $ of Forecasting Problem
of section 2 with the errorless initial conditions $\left( u_{0}^{\ast
},m_{0}^{\ast }\right) \in H_{2,0}^{k_{n}}\left( \Omega \right) ,$ see (\ref%
{5.30}) and (\ref{5.31}). Hence,%
\begin{equation}
L_{1}\left( u^{\ast },m^{\ast }\right) =L_{2}\left( u^{\ast },m^{\ast
}\right) =0.  \label{7.0}
\end{equation}%
Let a sufficiently small number $\delta \in \left( 0,1\right) $
characterizes the level of the error in the initial conditions (\ref{2.4}),%
\begin{equation}
\left\Vert u_{0}-u_{0}^{\ast }\right\Vert _{H^{2}\left( \Omega \right)
}<\delta ,\text{ }\left\Vert m_{0}-m_{0}^{\ast }\right\Vert _{H^{2}\left(
\Omega \right) }<\delta .  \label{7.1}
\end{equation}%
Since we seek solution $\left( u,m\right) \in \overline{B\left( R\right) }$
of our problem (see (\ref{5.31}) and (\ref{5.4})), then we assume that 
\begin{equation}
\left( u^{\ast },m^{\ast }\right) \in B^{\ast }\left( R\right) =\left\{ 
\begin{array}{c}
\left( u,m\right) \in H_{2,0}^{k_{n}}\left( Q_{T}\right) :\left\Vert \left(
u,m\right) \right\Vert _{H_{2}^{k_{n}}\left( Q_{T}\right) }<R, \\ 
u\left( x,0\right) =u_{0}^{\ast }\left( x\right) ,\text{ }m\left( x,0\right)
=m_{0}^{\ast }\left( x\right)%
\end{array}%
\right\} .  \label{7.2}
\end{equation}

Introduce functions 
\begin{equation}
\left. 
\begin{array}{c}
\widehat{u}_{0}\left( x,t\right) =u_{0}\left( x\right) ,\text{ }\widehat{m}%
_{0}\left( x,t\right) =m_{0}\left( x\right) , \\ 
\widehat{u}_{0}^{\ast }\left( x,t\right) =u_{0}^{\ast }\left( x\right) ,%
\widehat{m}_{0}^{\ast }\left( x,t\right) =m_{0}^{\ast }\left( x\right) .%
\end{array}%
\right.  \label{7.3}
\end{equation}%
By (\ref{7.3})%
\begin{equation}
\left. 
\begin{array}{c}
\widehat{u}_{0}\left( x,0\right) =u_{0}\left( x\right) ,\text{ }\widehat{m}%
_{0}\left( x,0\right) =m_{0}\left( x\right) , \\ 
\widehat{u}_{0}^{\ast }\left( x,0\right) =u_{0}^{\ast }\left( x\right) ,%
\text{ }\widehat{m}_{0}^{\ast }\left( x,0\right) =m_{0}^{\ast }\left(
x\right) .%
\end{array}%
\right.  \label{7.4}
\end{equation}%
It follows from (\ref{5.4}) and (\ref{7.1})-(\ref{7.4}) that it is natural
to assume that 
\begin{equation}
\left. 
\begin{array}{c}
\left( \widehat{u}_{0},\widehat{m}_{0}\right) \left( x,t\right) \in B\left(
R\right) ,\text{ } \\ 
\left( \widehat{u}_{0}^{\ast },\widehat{m}_{0}^{\ast }\right) \left(
x,t\right) \in B^{\ast }\left( R\right) .%
\end{array}%
\right.  \label{7.5}
\end{equation}

For each vector function $\left( u,m\right) \in B\left( R\right) $ we define
another vector function $\left( \overline{u},\overline{m}\right) $, 
\begin{equation}
\left( \overline{u},\overline{m}\right) \left( x,t\right) =\left( u,m\right)
\left( x,t\right) -\left( \widehat{u}_{0},\widehat{m}_{0}\right) \left(
x,t\right) .  \label{7.6}
\end{equation}%
Using (\ref{5.31}), (\ref{5.4}), (\ref{6.8}), (\ref{7.5}), (\ref{7.6}) and
triangle inequality, we obtain 
\begin{equation}
\left( \overline{u},\overline{m}\right) \in B_{0}\left( 2R\right) ,\text{ }%
\forall \left( u,m\right) \in B\left( R\right) .  \label{7.7}
\end{equation}%
In addition, we denote%
\begin{equation}
\left( \overline{u}^{\ast },\overline{m}^{\ast }\right) \left( x,t\right)
=\left( u^{\ast }-\widehat{u}_{0}^{\ast },m^{\ast }-\widehat{m}_{0}^{\ast
}\right) \left( x,t\right) .  \label{7.9}
\end{equation}%
Then, using (\ref{6.8}), (\ref{7.2})-(\ref{7.9}) and triangle inequality, we
obtain%
\begin{equation}
\left( \overline{u}^{\ast },\overline{m}^{\ast }\right) \in B_{0}\left(
2R\right) .  \label{7.10}
\end{equation}%
Next, by the first line of (\ref{7.5}) and triangle inequality%
\begin{equation}
\left( v+\widehat{u}_{0},s+\widehat{m}_{0}\right) \in B\left( 3R\right) ,%
\text{ }\forall \left( v,s\right) \in B_{0}\left( 2R\right) .  \label{7.11}
\end{equation}

Introduce a new functional $I_{\lambda ,\alpha }\left( v,s\right) ,$ 
\begin{equation}
\left. 
\begin{array}{c}
I_{\lambda ,\alpha }:B_{0}\left( 2R\right) \rightarrow \mathbb{R}, \\ 
I_{\lambda ,\alpha }\left( v,s\right) =J_{\lambda ,\alpha }\left( v+\widehat{%
u}_{0},s+\widehat{m}_{0}\right) ,\text{ }\forall \left( v,s\right) \in
B_{0}\left( 2R\right) .%
\end{array}%
\right.  \label{7.12}
\end{equation}%
Let $\lambda _{1}=\lambda _{1}\left( M,R,c,a,d,Q_{T}\right) \geq 1$ be the
value of the parameter $\lambda _{1}$ chosen in \cref{Theorem 5.1}. Define
the number $\lambda _{2}$ as 
\begin{equation}
\lambda _{2}=\lambda _{2}\left( M,3R,c,a,d,Q_{T}\right) =\lambda _{1}\left(
M,3R,c,a,d,Q_{T}\right) .  \label{7.13}
\end{equation}%
It follows from (\ref{7.11})-(\ref{7.13}) that the full analog of %
\cref{Theorem 5.1} is valid for the functional $I_{\lambda ,\alpha }\left(
v,s\right) $ for all $\lambda \geq \lambda _{2}.$ Let (\ref{2.00}) be valid.
Choose a sufficiently small number $\delta _{0}=\delta _{0}\left(
M,3R,c,d,\gamma ,Q_{T}\right) \in \left( 0,1\right) $ depending only on
listed parameters such as 
\begin{equation}
\frac{1}{2\ln \left( T+c\right) }\ln \left[ \left( \ln \left( \delta
_{0}^{-1/3}\right) \right) \right] \geq \lambda _{2}.  \label{7.14}
\end{equation}%
Define $\lambda =\lambda \left( \delta \right) $ as%
\begin{equation}
\lambda \left( \delta \right) =\frac{1}{2\ln \left( T+c\right) }\ln \left[
\left( \ln \left( \delta ^{-1/3}\right) \right) \right] ,\text{ }\forall
\delta \in \left( 0,\delta _{0}\right) .  \label{7.15}
\end{equation}

\begin{theorem} \label{Theorem 6.1}{Let conditions (\ref{2.00}), (\ref{5.31}), (\ref%
{7.1}), (\ref{7.2}), (\ref{7.5}), (\ref{7.14}) and (\ref{7.15}) hold.
Following (\ref{6.4}), set }%
\begin{equation}
\alpha =\alpha \left( \delta \right) =2e^{-\left( a-1\right) c^{\lambda
\left( \delta \right) }}.  \label{7.16}
\end{equation}%
{Let }$\left( v_{\min ,\lambda \left( \delta \right) },s_{\min ,\lambda
\left( \delta \right) }\right) \in \overline{B_{0}\left( 2R\right) }${\
be the minimizer of the functional }$I_{\lambda \left( \delta \right)
,\alpha }\left( v,s\right) ,${\ the existence and uniqueness of which
on the set }$\overline{B_{0}\left( 2R\right) }${\ are guaranteed by the
above mentioned analog of \cref{Theorem 5.1}. Let }$\left( \overline{u}^{\ast },%
\overline{m}^{\ast }\right) \in B_{0}\left( 2R\right) ${\ be the pair
of functions defined in (\ref{7.9}), also, see (\ref{7.10}). Then the
following error estimates are valid:}%
\begin{equation}
\left\Vert v_{\min ,\lambda \left( \delta \right) }-\overline{u}^{\ast
}\right\Vert _{H^{1,0}\left( Q_{\gamma T}\right) }+\left\Vert s_{\min
,\lambda \left( \delta \right) }-\overline{m}^{\ast }\right\Vert
_{H^{1,0}\left( Q_{\gamma T}\right) }\leq C_{3}\sqrt{\delta }.  \label{7.17}
\end{equation}

{Let }%
\begin{equation}
\widetilde{v}_{\min ,\lambda \left( \delta \right) }=v_{\min ,\lambda \left(
\delta \right) }+\widehat{u}_{0},\widetilde{s}_{\min ,\lambda \left( \delta
\right) }=s_{\min ,\lambda \left( \delta \right) }+\widehat{m}_{0},
\label{7.18}
\end{equation}%
{see (\ref{7.3}). Then} 
\begin{equation}
\left\Vert \widetilde{v}_{\min ,\lambda \left( \delta \right) }-u^{\ast
}\right\Vert _{H^{1,0}\left( Q_{\gamma T}\right) }\leq C_{3}\sqrt{\delta },%
\text{ }\left\Vert \widetilde{s}_{\min ,\lambda \left( \delta \right)
}-m^{\ast }\right\Vert _{H^{1,0}\left( Q_{\gamma T}\right) }\leq C_{3}\sqrt{%
\delta }.  \label{7.19}
\end{equation}
\end{theorem}

Thus, estimates (\ref{7.17})-(\ref{7.19}) imply that, when minimizing the
functional \newline
$I_{\lambda,\alpha}\left( v,s\right)$ defined in (\ref{7.12}), we obtain
good approximations for both: the pair of functions $\left( \overline{u}%
^{\ast },\overline{m}^{\ast }\right) ,$ generated by the ideal solution $%
\left( u^{\ast },m^{\ast }\right) $ via (\ref{7.9}) and for this ideal
solution as well.

\begin{proof}[Proof of \cref{Theorem 6.1}] Without yet specifying\textbf{\ }the values
of $\delta $ and $\lambda \left( \delta \right) ,$ we set $\lambda _{2}$ as
in (\ref{7.13}), $\lambda \geq \lambda _{2}$ and assume (\ref{7.16}). Then,
using (\ref{7.9}) and the analog of (\ref{6.5}) for the
functional $I_{\lambda ,\alpha },$ we obtain 
\begin{equation}
\left. 
\begin{array}{c}
e^{2ac^{\lambda }}\left( \left( T+c\right) /c\right) ^{\lambda }\exp \left[
-2\left( T\left( 1-\gamma \right) +c\right) ^{\lambda }\right] \times \\ 
\times \left[ I_{\lambda ,\alpha }\left( \overline{u}^{\ast },\overline{m}%
^{\ast }\right) -I_{\lambda ,\alpha }\left( v_{\min ,\lambda },s_{\min
,\lambda }\right) \right] - \\ 
-e^{2ac^{\lambda }}\left( \left( T+c\right) /c\right) ^{\lambda }\exp \left[
-2\left( T\left( 1-\gamma \right) +c\right) ^{\lambda }\right] \times \\ 
\times \left[ I_{\lambda ,\alpha }^{\prime }\left( v_{\min ,\lambda
},s_{\min ,\lambda }\right) ,\left( \overline{u}^{\ast }-v_{\min ,\lambda },%
\overline{m}^{\ast }-s_{\min ,\lambda }\right) \right] \geq \\ 
\geq C_{3}\left( \left\Vert \overline{u}^{\ast }-v_{\min ,\lambda
}\right\Vert _{H^{1,0}\left( Q_{\gamma T}\right) }^{2}+\left\Vert \overline{m%
}^{\ast }-s_{\min ,\lambda }\right\Vert _{H^{1,0}\left( Q_{\gamma T}\right)
}^{2}\right) .%
\end{array}%
\right.  \label{7.20}
\end{equation}%
By (\ref{6.6}) 
\begin{equation*}
-\left[ I_{\lambda ,\alpha }^{\prime }\left( v_{\min ,\lambda },s_{\min
,\lambda }\right) ,\left( \overline{u}^{\ast }-v_{\min ,\lambda },\overline{m%
}^{\ast }-s_{\min ,\lambda }\right) \right] \leq 0.
\end{equation*}
Also, obviously $-I_{\lambda ,\alpha }\left( v_{\min ,\lambda },s_{\min
,\lambda }\right) \leq 0.$ Hence, (\ref{7.20}) implies%
\begin{equation}
\left. 
\begin{array}{c}
e^{2ac^{\lambda }}\left( \left( T+c\right) /c\right) ^{\lambda }\exp \left[
-2\left( T\left( 1-\gamma \right) +c\right) ^{\lambda }\right] I_{\lambda
,\alpha }\left( \overline{u}^{\ast },\overline{m}^{\ast }\right) \geq \\ 
\geq \left\Vert \overline{u}^{\ast }-v_{\min ,\lambda }\right\Vert
_{H^{1,0}\left( Q_{\gamma T}\right) }^{2}+\left\Vert \overline{m}^{\ast
}-s_{\min ,\lambda }\right\Vert _{H^{1,0}\left( Q_{\gamma T}\right) }^{2},%
\text{ }\forall \lambda \geq \lambda _{2}.%
\end{array}%
\right.  \label{7.21}
\end{equation}

We now estimate the left hand side of (\ref{7.21}) from the above. By (\ref%
{5.4}), (\ref{6.8}), (\ref{6.9}), (\ref{7.5}), (\ref{7.9})-(\ref{7.12}) and (%
\ref{7.16}) 
\begin{equation}
\left. 
\begin{array}{c}
I_{\lambda ,\alpha }\left( \overline{u}^{\ast },\overline{m}^{\ast }\right)
=J_{\lambda ,\alpha }\left( \overline{u}^{\ast }+\widehat{u}_{0},\overline{m}%
^{\ast }+\widehat{m}_{0}\right) \leq \\ 
\leq e^{-2ac^{\lambda }}\dint\limits_{Q_{T}}\left[ \left( L_{1}\left( 
\overline{u}^{\ast }+\widehat{u}_{0},\overline{m}^{\ast }+\widehat{m}%
_{0}\right) \right) +qd\left( L_{2}\left( \overline{u}^{\ast }+\widehat{u}%
_{0},\overline{m}^{\ast }+\widehat{m}_{0}\right) \right) ^{2}\right] \varphi
_{\lambda }^{2}dxdt+ \\ 
+C_{3}e^{-\left( a-1\right) c^{\lambda }}.%
\end{array}%
\right.  \label{7.22}
\end{equation}%
Using (\ref{7.4}) and (\ref{7.9}), we obtain%
\begin{equation}
\left( u^{\ast }-\widehat{u}_{0}^{\ast }+\widehat{u}_{0},m^{\ast }-\widehat{m%
}_{0}^{\ast }+\widehat{m}_{0}\right) =\left( u^{\ast },m^{\ast }\right)
+\left( \widehat{u}_{0}-\widehat{u}_{0}^{\ast },\widehat{m}_{0}-\widehat{m}%
_{0}^{\ast }\right) .  \label{7.23}
\end{equation}%
It follows from (\ref{2.1}), (\ref{2.2}), (\ref{7.0}), (\ref{7.1}), (\ref%
{7.3}) and (\ref{7.23}) that 
\begin{equation}
\left. 
\begin{array}{c}
\left\Vert L_{i}\left( \overline{u}^{\ast }+\widehat{u}_{0},\overline{m}%
^{\ast }+\widehat{m}_{0}\right) \right\Vert _{L_{2}\left( Q_{T}\right) }^{2}=
\\ 
=\left\Vert L_{i}\left( \left( u^{\ast },m^{\ast }\right) +\left( \widehat{u}%
_{0}-\widehat{u}_{0}^{\ast },\widehat{m}_{0}-\widehat{m}_{0}^{\ast }\right)
\right) \right\Vert _{L_{2}\left( Q_{T}\right) }^{2}\leq C_{3}\delta ^{2},%
\text{ }i=1,2.%
\end{array}%
\right.  \label{7.230}
\end{equation}%
Hence, using (\ref{5.7}) and (\ref{7.22}), we obtain%
\begin{multline} \label{7.24}
\left\Vert v_{\min ,\lambda }-\overline{u}^{\ast }\right\Vert
_{H^{1,0}\left( Q_{\gamma T}\right) }^{2}+\left\Vert s_{\min ,\lambda }-%
\overline{m}^{\ast }\right\Vert _{H^{1,0}\left( Q_{\gamma T}\right) }^{2}\leq \\
\leq C_{3}\exp \left( 3\left( T+c\right) ^{\lambda }\right) \delta ^{2}+ C_{3}\left( \frac{T+c}{c}\right) ^{\lambda }e^{\left( a+1\right) c^{\lambda
}}\exp \left[ -2\left( T\left( 1-\gamma \right) +c\right) ^{\lambda }\right]
.
\end{multline}

Until now we have not used conditions (\ref{7.14}) and (\ref{7.15}). Now we
start to use them. Hence, elementary calculations lead to: 
\begin{equation}
\exp \left( 3\left( T+c\right) ^{\lambda \left( \delta \right) }\right)
\delta ^{2}=\delta ,\text{ }\forall \delta \in \left( 0,\delta _{0}\right) ,
\label{7.25}
\end{equation}%
\begin{equation*}
C_{3}\left( \frac{T+c}{c}\right) ^{\lambda \left( \delta \right) }e^{\left(
a+1\right) c^{\lambda \left( \delta \right) }}\exp \left[ -2\left( T\left(
1-\gamma \right) +c\right) ^{\lambda \left( \delta \right) }\right] \leq
\end{equation*}%
\begin{equation}
\leq C_{3}\exp \left[ -\left( T\left( 1-\gamma \right) +c\right) ^{\lambda
\left( \delta \right) }\right] \leq C_{3}\delta ,\text{ }\forall \delta \in
\left( 0,\delta _{0}\right) .  \label{7.26}
\end{equation}

The first target estimate (\ref{7.17}) of this theorem follows immediately
from (\ref{7.24})-(\ref{7.26}).\ Next, using (\ref{7.17}), (\ref{7.18}) and
the same considerations as the ones in the derivations of (\ref{7.23}) and (%
\ref{7.230}), we obtain the second target estimate (\ref{7.19}).
\end{proof}

\section{Global Convergence of the Gradient Descent Method}

\label{sec:7}

We assume in this section that conditions of \cref{Theorem 6.1} are in
place. Suppose that%
\begin{equation}
\left( u^{\ast },m^{\ast }\right) \in B^{\ast }\left( \frac{R}{3}-\delta
\right) ,\mbox{ }\frac{R}{3}-\delta >0.  \label{8.1}
\end{equation}%
We also assume that for $\lambda =\lambda \left( \delta \right) \geq \lambda
_{2}$ the minimizer $\left( u_{\min ,\lambda },m_{\min ,\lambda }\right) $
of the functional $J_{\lambda ,\alpha }\left( u,m\right) $ on the set $%
\overline{B\left( R\right) },$ which was found in \cref{Theorem 5.1}, is
such that 
\begin{equation}
\left( \widetilde{v}_{\min ,\lambda },\widetilde{s}_{\min ,\lambda }\right)
=\left( u_{\min ,\lambda },m_{\min ,\lambda }\right) \in B\left( \frac{R}{3}%
\right) .  \label{8.2}
\end{equation}%
Formulas (\ref{7.17})-(\ref{7.19}) and (\ref{8.1}) indicate that assumption (%
\ref{8.2}) is reasonable.

We construct now the gradient descent method of the minimization of the
functional $J_{\lambda ,\alpha }.$ Consider an arbitrary pair of functions%
\begin{equation}
\left( u^{0},m^{0}\right) \in B\left( \frac{R}{3}\right) .  \label{8.3}
\end{equation}%
The iterative sequence of the gradient descent method with the step size $%
\xi >0$ is: 
\begin{equation}
\left( u^{n},m^{n}\right) =\left( u^{n-1},m^{n-1}\right) -\xi J_{\lambda
,\alpha }^{\prime }\left( u^{n-1},m^{n-1}\right) ,\mbox{ }n=1,2,...
\label{8.4}
\end{equation}%
By \cref{Theorem 6.1} $J_{\lambda ,\alpha }^{\prime }\left( u,m\right) \in
H_{2,0,0}^{k_{n}}\left( Q_{T}\right) $, see (\ref{5.31}). Hence, all pairs $%
\left( u^{n},m^{n}\right) $ have the same initial values $\left(
u^{n},m^{n}\right) \left( x,0\right) =\left( u^{0},m^{0}\right) \left(
x,0\right) =\left( u_{0}\left( x\right) ,m_{0}\left( x\right) \right) .$

\begin{theorem} \label{Theorem 7.1}{Let conditions of \cref{Theorem 6.1} as well as
conditions (\ref{8.1})-(\ref{8.4}) hold. Then there exists a number }$\xi
_{0}\in \left( 0,1\right) ${\ such that for any }$\xi \in \left( 0,\xi
_{0}\right) ${\ there exists a number }$\rho =\rho \left( \xi \right)
\in \left( 0,1\right) ${\ such that for all }$n\geq 1$%
\begin{equation}
\left. 
\begin{array}{c}
\left( u^{n},m^{n}\right) \in B\left( R\right) , \\ 
\left\Vert u^{n}-u^{\ast }\right\Vert _{_{H^{1,0}\left( Q_{\gamma T}\right)
}}+\left\Vert m^{n}-m^{\ast }\right\Vert _{_{H^{1,0}\left( Q_{\gamma
T}\right) }}\leq \\ 
\leq C_{3}\sqrt{\delta }+C_{3}\rho ^{n}\left( \left\Vert u_{\min ,\lambda
}-u^{0}\right\Vert _{H^{k_{n}}\left( Q_{T}\right) }+\left\Vert m_{\min
,\lambda }-m^{0}\right\Vert _{H^{k_{n}}\left( Q_{T}\right) }\right) .%
\end{array}%
\right.  \label{8.6}
\end{equation}%
{\ }
\end{theorem}

\begin{proof} The first line of (\ref{8.6}), existence of numbers $\xi
_{0} $ and $\rho \left( \xi \right) $ as well as estimate 
\begin{equation}
\hspace{-1cm}\left. 
\begin{array}{c}
\left\Vert u_{\min ,\lambda }-u^{n}\right\Vert _{_{H^{k_{n}}\left( Q_{\gamma
T}\right) }}+\left\Vert m_{\min ,\lambda }-m^{n}\right\Vert
_{_{H^{k_{n}}\left( Q_{\gamma T}\right) }}\leq \\ 
\leq C_{3}\rho ^{n}\left( \left\Vert u_{\min ,\lambda }-u^{0}\right\Vert
_{H^{k_{n}}\left( Q_{T}\right) }+\left\Vert m_{\min ,\lambda
}-m^{0}\right\Vert _{H^{k_{n}}\left( Q_{T}\right) }\right) .%
\end{array}%
\right.  \label{8.7}
\end{equation}%
follow immediately from (\ref{8.1})-(\ref{8.4}) and \cite[Theorem 6]{SAR}.
Next, triangle inequality, (\ref{7.18}), (\ref{7.19}), (\ref{8.2}) and (\ref%
{8.7}) lead to the desired result.
\end{proof}

\begin{remark}
\begin{enumerate}
\item Since $R>0$\ is an arbitrary number and the pair $\left(
u_{0},m_{0}\right) $\ in (\ref{8.3}) is an arbitrary point of the set $%
B\left( R/3\right) ,$\ then \cref{Theorem 7.1} implies the global
convergence of the iterative procedure (\ref{8.4}), see \cref{def:1.1} in
section 1.

\item {{Even though the above theory requires sufficiently large
values of the parameter $\lambda$ our numerical studies of section
8 demonstrate that $\lambda =2$ is sufficient. Furthermore, the
numerical experience of all previous publications on the convexification
method demonstrates that values $\lambda \in \left[ 1,5\right] $
are sufficient \cite{KL,MFGbook,MFG7,MFGCAMWA}. This is similar with an asymptotic theory.
Indeed, such a theory typically states that if a certain parameter $X$
 is sufficiently large, then a certain formula $Y$ is
sufficiently accurate. However, in a computational practice only
results of numerical experiments can tell one which exactly values of $X$ ensure a good accuracy of $Y$.}}
\end{enumerate}
\end{remark}

\section{Numerical Studies}

\label{sec:8}

Numerical tests were performed to test the performance of the
convexification method for the forecasting problem. {While our theoretical
results hold for any dimension $n \geq 1$, we restrict attention to the case 
$n=1$ in our numerical experiments. In terms of public sentiment
forecasting, this special case may be relevant to settings in which there is
a single dominant direction of variation in sentiment. For example, in our
forthcoming work \cite{Kexper} we consider the public sentiment (positive
and negative) regarding interventions against COVID-19 as a one-dimensional
quantity. In the numerical studies presented here,} we chose our other
parameters for these numerical experiments as follows. We used $a=1.1$, $c=3$%
, and the regularization parameter $\alpha =10^{-5}$. These values of
parameters are chosen by a trial and error process to ensure that the method
works well. We used $\lambda =2$ following the same procedure as in previous
publications of the method, see \cite{MFG7,MFGCAMWA}, for example. We test
the method on the following model: 
\begin{gather}
u_{t}+u_{xx}+\frac{1}{2}u_{x}^{2}+\int_{\Omega }K(x,y)m(y,t)dy+f(x,t)m=0,
\label{num:1} \\
m_{t}-m_{xx}+\partial _{x}(mu_{x})=0,  \label{num:2} \\
u_{x}=m_{x}=0,\quad \text{on $\{\pm 1\}\times \lbrack 0,T]$},  \label{num:3}
\end{gather}%
i.e. $r(x,t)=-1$ in equation (\ref{2.1}). Both the spatial dimension $[-1,1]$
and temporal dimension $[0,T]$ are evenly discretized with step size $0.1$.

The minimization of the functional $J_{\lambda ,\alpha }(u,m)$ is performed
using the \textbf{fmincon} function in MATLAB. This function requires a
starting point, which is simply chosen as 
\begin{equation*}
u_{\text{start}}(x,t)=u(x,0),\quad m_{\text{start}}(x,t)=m(x,0),\quad
\forall t\in \lbrack 0,T].
\end{equation*}%
The algorithm terminates when the \textit{first order optimality} is less
than $10^{-5}$. The first order optimality for \textbf{fmincon} at $(u,m)$
is defined as the ratio 
\begin{equation*}
\frac{\Vert \nabla _{\text{proj}}J_{\lambda ,\alpha }(u,m)\Vert}{\Vert
\nabla J_{\lambda ,\alpha }(u_{\text{start}},m_{\text{start}})\Vert},
\end{equation*}%
where $\nabla _{\text{proj}}J_{\lambda ,\alpha }(u,m)$ is the projected
gradient of the functional $J_{\lambda ,\alpha }(u,m)$ onto the space of the
constraints. The constraints include the Neumann boundary conditions (\ref%
{num:3}) and the data $u(x,0)$ and $m(x,0)$, see (\ref{5.4}). We refer to
the documentation of \textbf{fmincon} for more details on the algorithm.

In all tests, $3\%$ of noise is added to the initial data according to the
following noise model: $\text{data} = \text{true data} + \text{noise}$,
where $\text{noise} = 0.03 \cdot \|\text{true data}\|_2 \cdot \text{rand}$.
In this noise model, data/true data are arrays representing the discretized
version of the corresponding $u(x,0)$ and $m(x,0)$, rand is a random array
whose entries are uniformly distributed in $[-1, 1]$, and $\| \cdot \|_2$ is
the vector 2-norm.

We test the convexification method in two cases:

\textbf{The ideal case:} The true solution can be found to compare with the
predicted solution by the convexification method. It is hard to construct a
precise solution of problem (\ref{num:1})-(\ref{num:3}). Nevertheless, we
need to verify somehow the accuracy of our numerical technique. Hence, we
proceed similarly with \cite{MFG7,MFGCAMWA}. To be specific, the true
solution is found as follows:

\begin{enumerate}
\item Choose a function $u(x,t)$ satisfying (\ref{num:3}) and an initial
condition $m_0(x)$ for $m(x,t)$.

\item Solve for $m(x,t)$ using (\ref{num:2}), (\ref{num:3}), and $m(x,0) =
m_0(x)$.

\item Assuming that $m(x,t)$ does not vanish at any point $(x,t)$, we set
the function $f(x,t)$ in equation (\ref{num:1}) as: 
\begin{equation*}
f(x,t)=-\frac{1}{m(x,t)}\left[ u_{t}+u_{xx}+\frac{1}{2}u_{x}^{2}+\int_{%
\Omega }K(x,y)m(y,t)dy\right] .
\end{equation*}
\end{enumerate}

Then, we will have a MFG system that admits the solution $(u(x,t), m(x,t))$.
We extract the initial values $u(x,0)$ and $m(x,0)$ from the found solution
as data.

\textbf{A more realistic case:} This is when we set in (\ref{num:1}) $f=0$,
and thus, the true solution cannot easily be found. An important point of
our study of this case is that we will evaluate the relative cost without
the CWF and regularization terms at the solution: to see how well the
minimizer of the functional $J_{\lambda ,\alpha }$ satisfies the MFG system (%
\ref{num:1}), (\ref{num:2}). This relative cost at the time $t$ is defined
as: 
\begin{equation}
F(t)=\left[ \frac{\mathop{\displaystyle \int}\limits_{-1}^{1}\left[
L_{1}(u(x,t),m(x,t))^{2}+L_{2}(u(x,t),m(x,t))^{2}\right] dx}{%
\mathop{\displaystyle \int}\limits_{-1}^{1}\left[ u^{2}(x,0)+m^{2}(x,0)%
\right] dx}\right]^{\frac{1}{2}} .  \label{num:rel-cost}
\end{equation}%
We note that the CWF is not involved in integrals in (\ref{num:rel-cost}), \
neither is the regularization term. We also point out that $(u,m)$ in (\ref%
{num:rel-cost}) is the minimizer of our target functional $J_{\lambda
,\alpha }$.

\begin{remark}
We observe numerically that the convexification method can predict the
solution of the MFG system accurately up to time $T=1$. See Figures \ref%
{fig:1-u-extended} for an example of the predicted solution $u(x,t)$ over
time for $t\in \lbrack 0,2]$ in the ideal case. The relative cost is also
shown in Figure \ref{fig:1-cost-extended}. We can see that both the
predicted solution and the relative cost behave badly for $t>1$, even for
the ideal case. From the theoretical point of view, this is to be expected
since the solution $u(x,t)$ of the MFG system with the initial data at $%
\left\{ t=0\right\} $ eventually blows up as $t$ increases, as it is well
known even for the simplistic case of the heat equation with the reversed
time $u_{t}+u_{xx}=0$.
On the other hand, the convexification method seeks a solution with a
bounded norm. This means that the minimizer of the functional (\ref{5.6}),
which we use in (\ref{num:rel-cost}), actually represents a compromise
between the stability for the function $m(x,t)$ and the instability for the
function $u(x,t)$. However, this compromise holds only for $t\in \lbrack
0,1],$ and it is broken after $t=1$. Therefore, from now on, we will only
consider the time interval $t\in \lbrack 0,1]$ for the numerical tests.
\end{remark}

\begin{figure}[ht!]
\centering
\subfloat[$t=0$]{\includegraphics[width=0.3\textwidth]{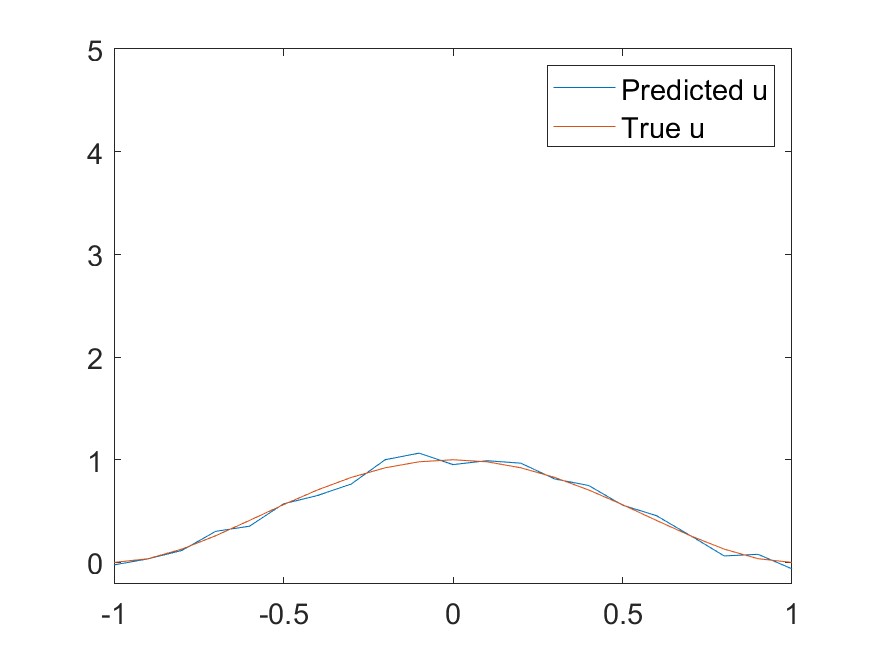}} %
\subfloat[$t=0.6$]{\includegraphics[width=0.3\textwidth]{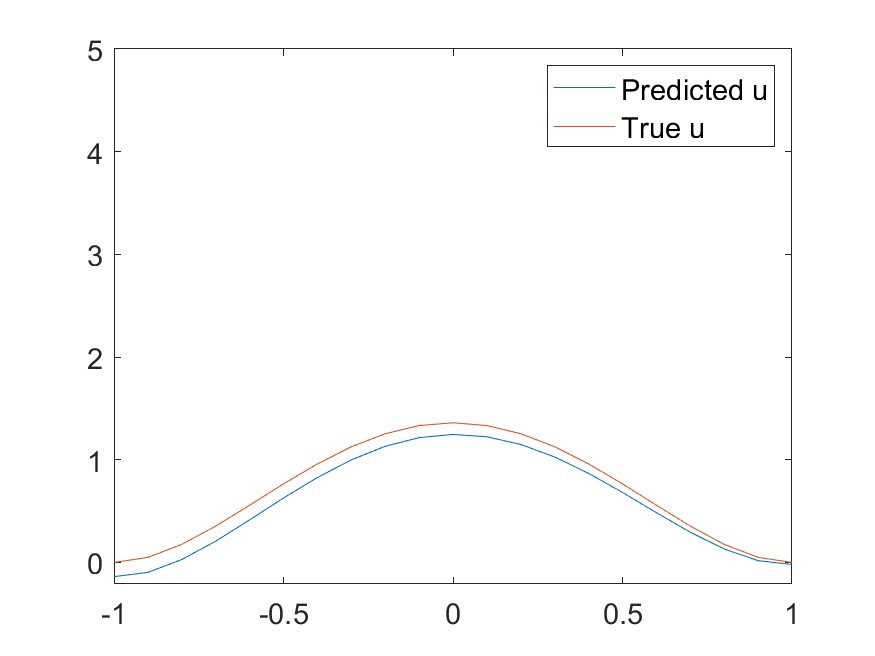}} %
\subfloat[$t=1$]{\includegraphics[width=0.3\textwidth]{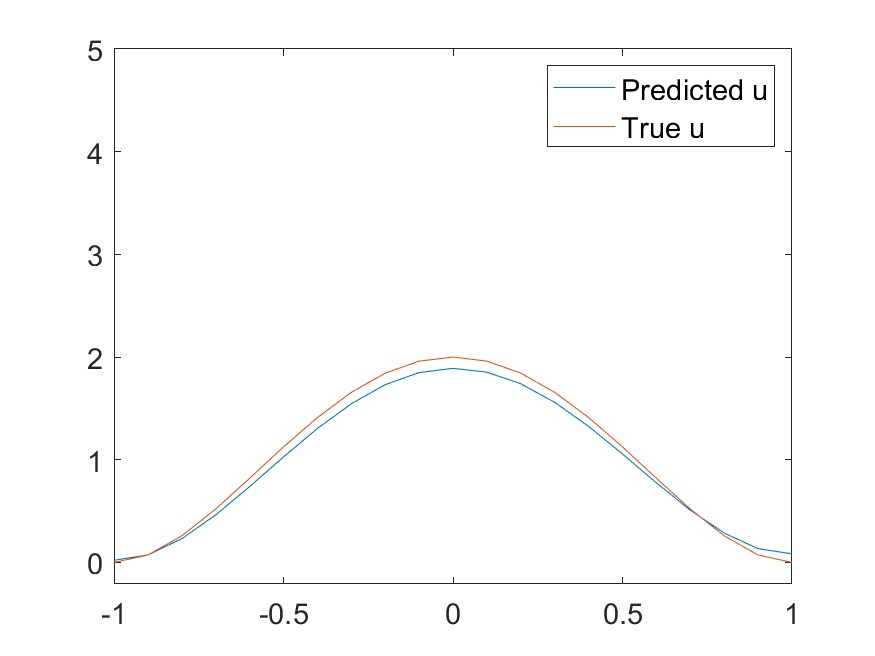}} \newline
\subfloat[$t=1.2$]{\includegraphics[width=0.3\textwidth]{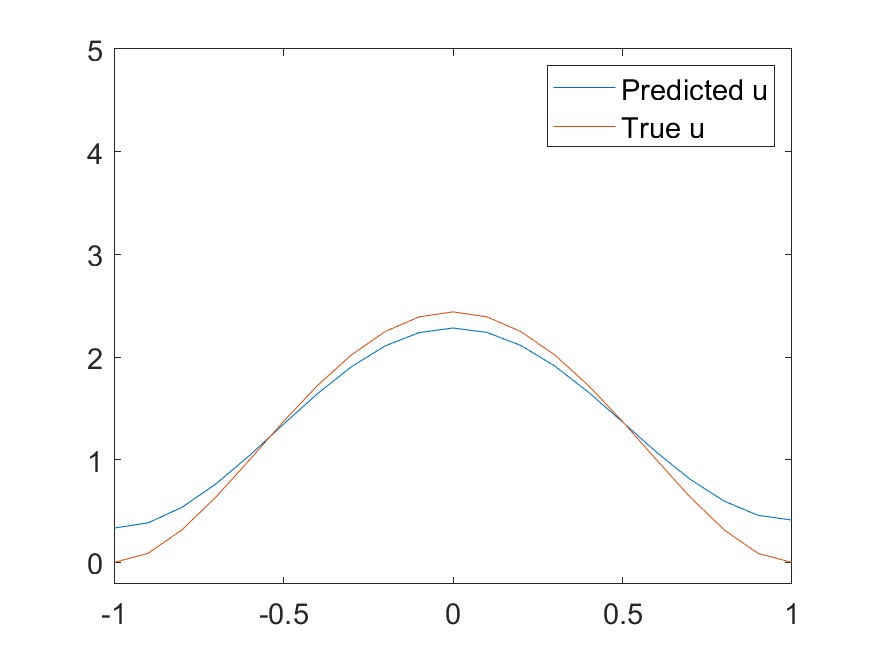}} %
\subfloat[$t=1.6$]{\includegraphics[width=0.3\textwidth]{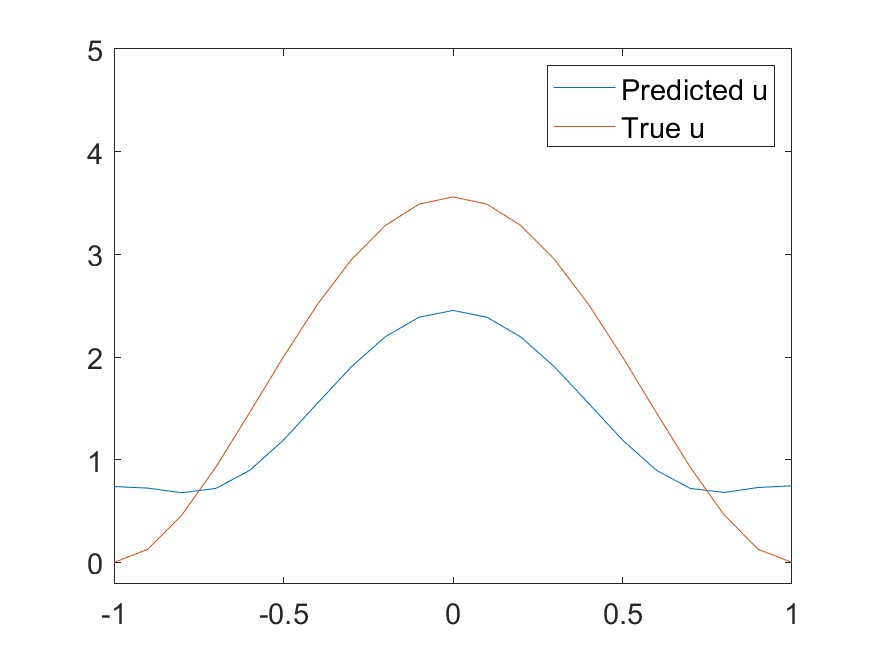}} %
\subfloat[$t=2$]{\includegraphics[width=0.3\textwidth]{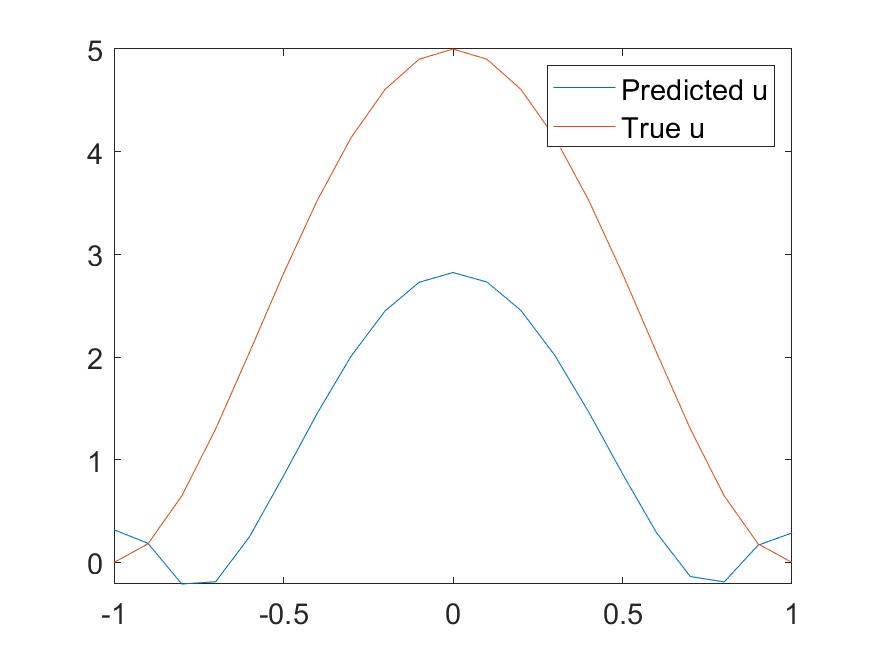}}
\caption{Test 1.1 (an ideal case): Predicted and true $u(x,t)$ for the
extended time interval $t \in [0,2]$.}
\label{fig:1-u-extended}
\end{figure}

\begin{figure}[ht!]
\centering
\includegraphics[width=0.3\textwidth]{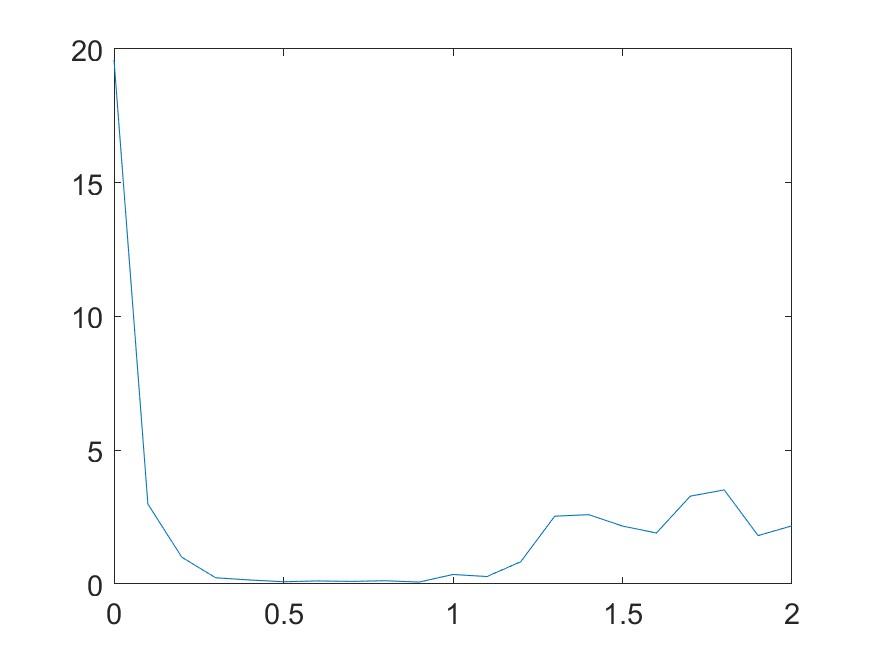}
\caption{Test 1.1 (an ideal case): Relative cost (\protect\ref{num:rel-cost}%
) for the extended time interval $t \in [0,2]$.}
\label{fig:1-cost-extended}
\end{figure}

\begin{remark}
In these numerical studies we evaluated the system with two simple choices for the kernel $K(x,y)$. In particular, we considered the cases when the running payoff function for individual players is equal to the expected value of the sentiment density (corresponding to $K(x,y) \equiv 1$) or the negative of the expected value of the sentiment density (corresponding to $K(x,y) \equiv -1$). These kernels correspond to reward structures in which individuals receive payoffs based on the mean public sentiments held by the population. Numerical results for $K(x,y)\equiv 1$ and for $K(x,y)\equiv -1$ (not shown) are similar. Hence, we will present from now on the results for $K(x,y)\equiv 1$.
\end{remark}

\subsection{The ideal case}

\textbf{Test 1.1:} We choose $u(x,t) = (x^2-1)^2(t^2+1)$ and $m(x,0)=\text{%
exp}\left( 1/(x^2-1) \right)+0.28$. The constant $0.28$ is to ensure that
the integral of $m(x,0)$ is approximately $1$. These functions are actually
the same ones that we choose for Figures \ref{fig:1-u-extended} and \ref%
{fig:1-cost-extended} when showcasing the method's behavior for extended
time. The true and predicted solutions for $u(x,t)$ and $m(x,t)$ are shown
in Figures \ref{fig:1-u} and \ref{fig:1-m}, respectively.

\begin{figure}[ht!]
\centering
\subfloat[$t=0$]{\includegraphics[width=0.3\textwidth]{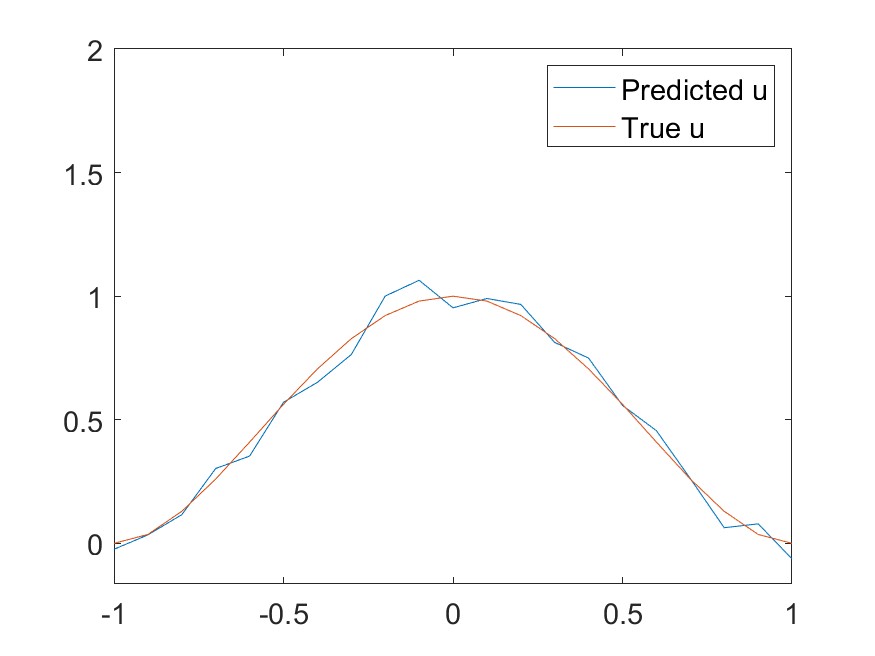}} %
\subfloat[$t=0.6$]{\includegraphics[width=0.3\textwidth]{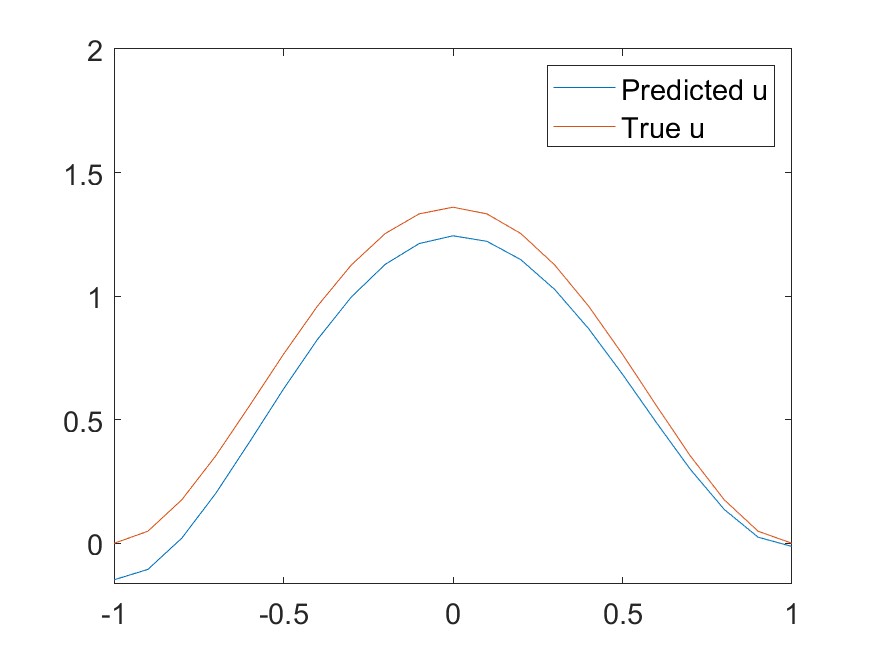}} %
\subfloat[$t=1$]{\includegraphics[width=0.3\textwidth]{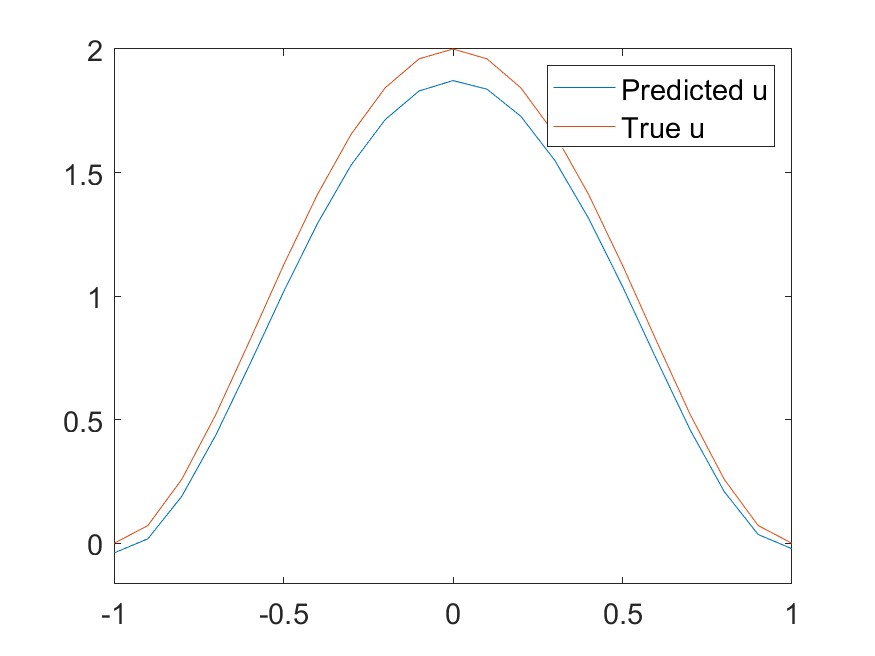}}
\caption{Test 1.1 (an ideal case): True and predicted $u(x,t)$ over time for 
$u(x,t) = (x^2-1)^2(t^2+1)$ and $m(x,0)=\text{exp}\left( 1/(x^2-1)
\right)+0.28$.}
\label{fig:1-u}
\end{figure}

\begin{figure}[ht!]
\centering
\subfloat[$t=0$]{\includegraphics[width=0.3\textwidth]{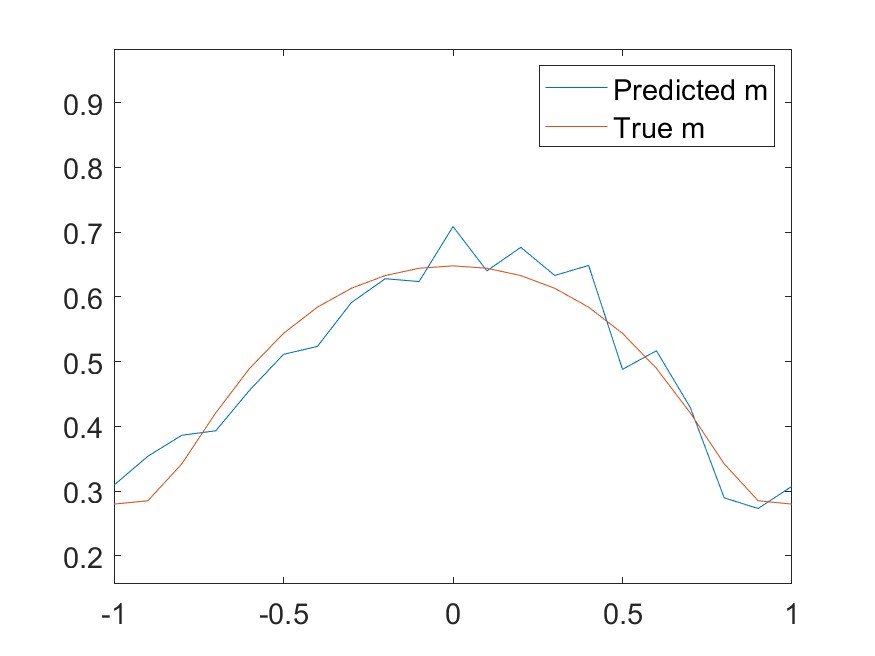}} %
\subfloat[$t=0.6$]{\includegraphics[width=0.3\textwidth]{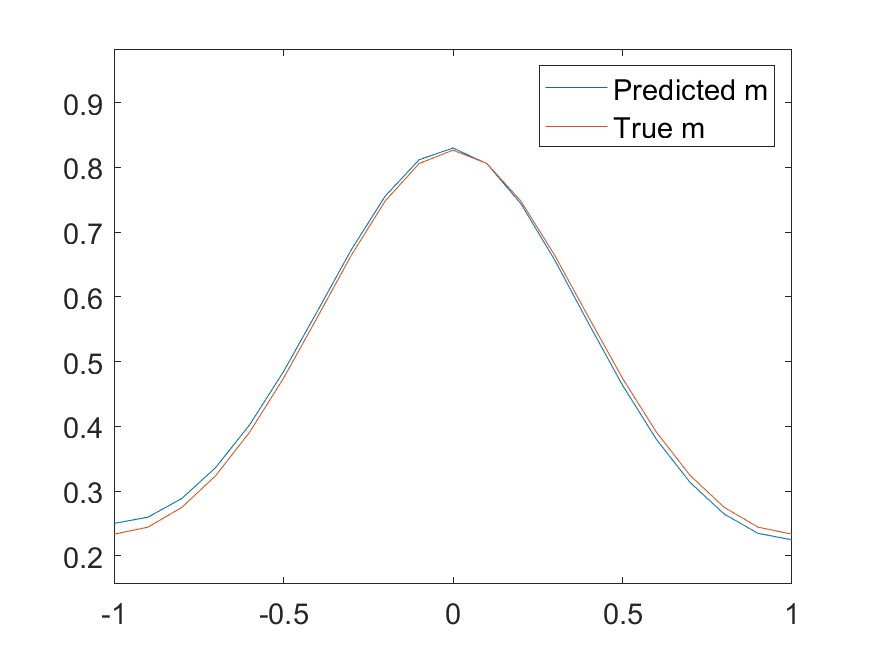}} %
\subfloat[$t=1$]{\includegraphics[width=0.3\textwidth]{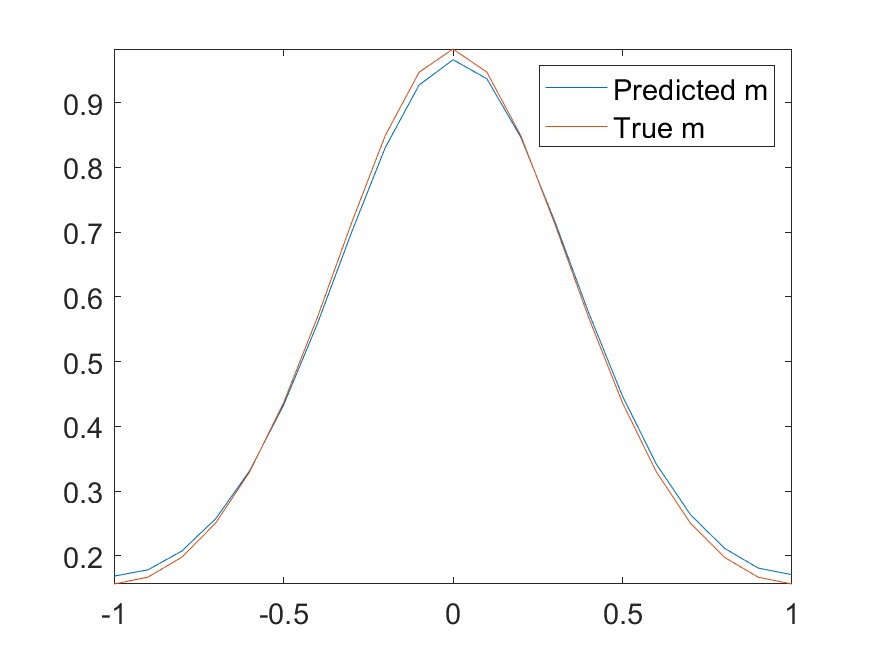}}
\caption{Test 1.1 (an ideal case): True and predicted $m(x,t)$ over time for 
$u(x,t) = (x^2-1)^2(t^2+1)$ and $m(x,0)=\text{exp}\left( 1/(x^2-1)
\right)+0.28$.}
\label{fig:1-m}
\end{figure}

\textbf{Test 1.2:} We choose $u(x,t)=0.1\cos {(2\pi x)}(t+1)$ and $%
m(x,0)=0.5 $. This test considers a more oscillatory $u(x,t)$ while $m(x,t)$
is constant at the initial time. The true and predicted solutions for $%
u(x,t) $ and $m(x,t)$ are shown in Figures \ref{fig:2-u} and \ref{fig:2-m},
respectively.

\begin{figure}[ht!]
\centering
\subfloat[$t=0$]{\includegraphics[width=0.3\textwidth]{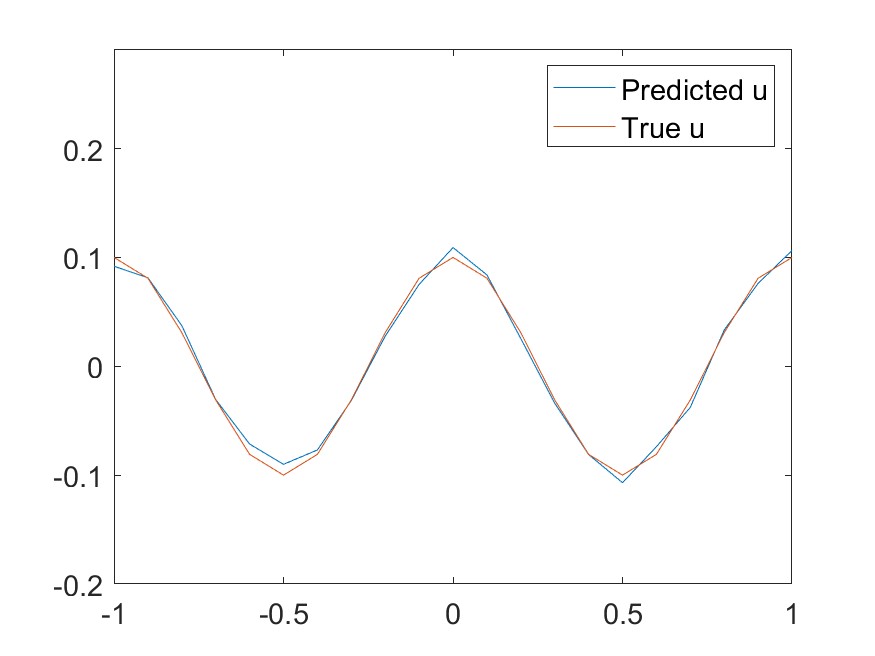}} %
\subfloat[$t=0.6$]{\includegraphics[width=0.3\textwidth]{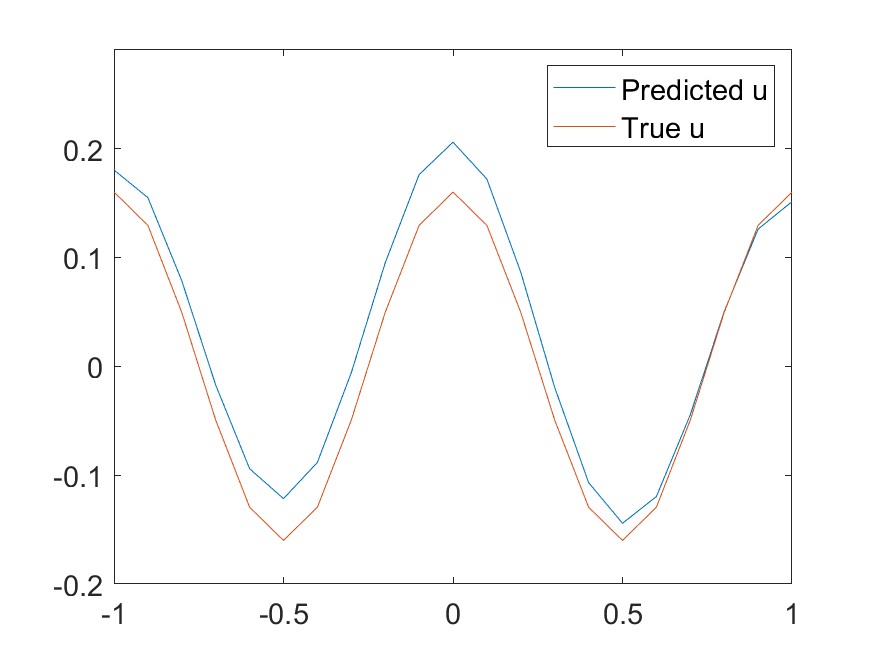}} %
\subfloat[$t=1$]{\includegraphics[width=0.3\textwidth]{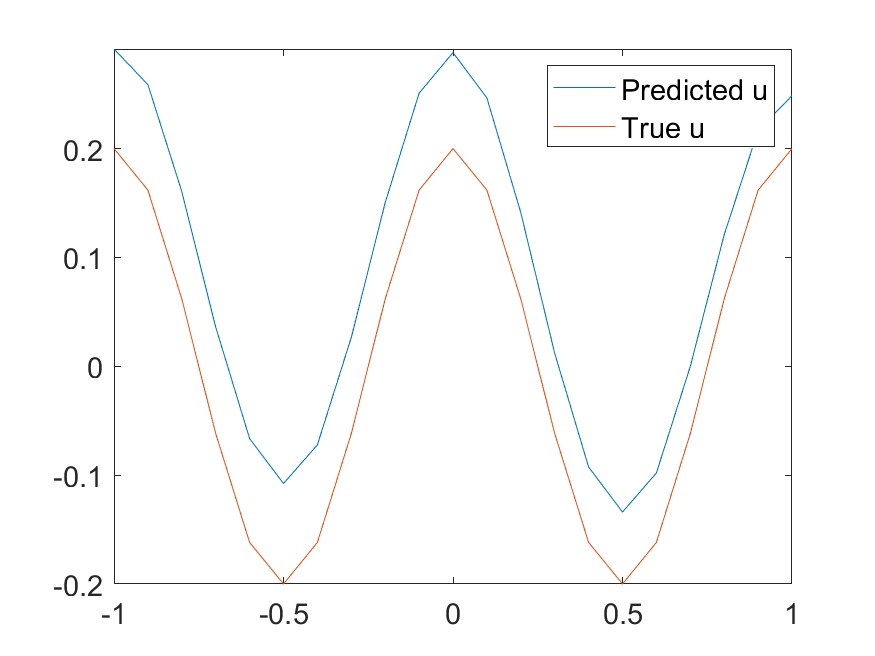}}
\caption{Test 1.2 (an ideal case): True and predicted $u(x,t)$ over time for 
$u(x,t) = 0.1\cos{(2\protect\pi x)}(t+1)$ and $m(x,0)=0.5$.}
\label{fig:2-u}
\end{figure}

\begin{figure}[ht!]
\centering
\subfloat[$t=0$]{\includegraphics[width=0.3\textwidth]{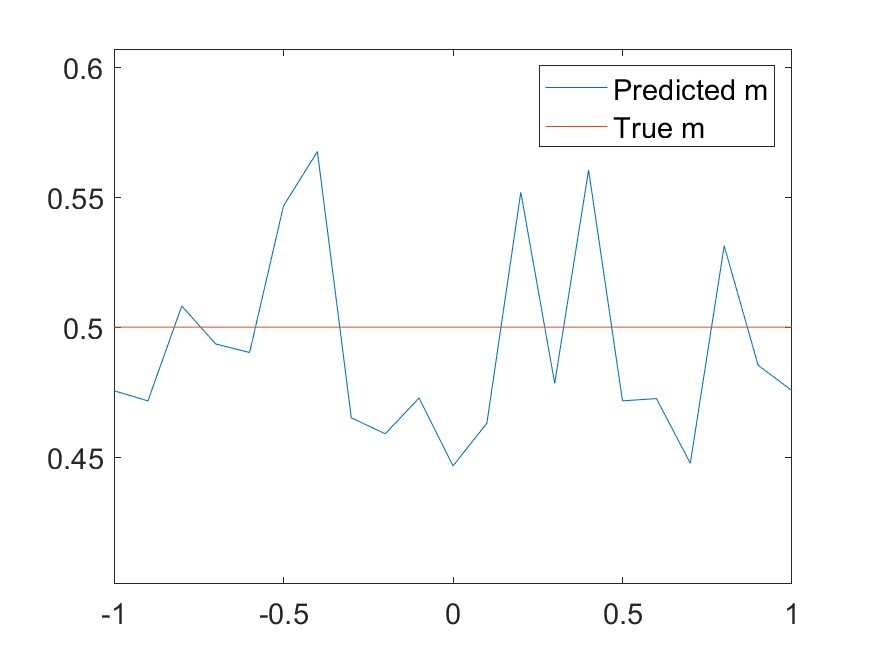}} %
\subfloat[$t=0.6$]{\includegraphics[width=0.3\textwidth]{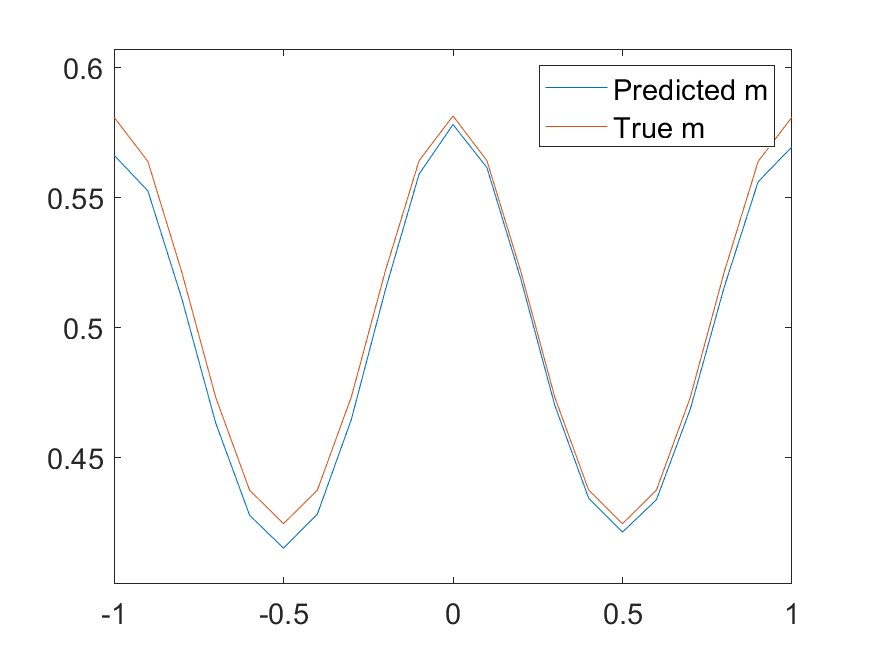}} %
\subfloat[$t=1$]{\includegraphics[width=0.3\textwidth]{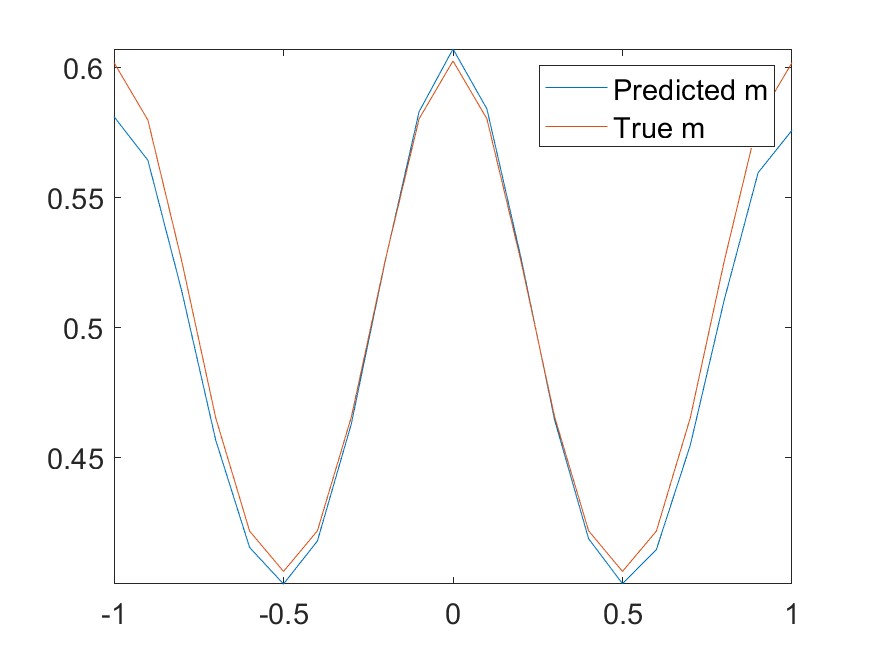}}
\caption{Test 1.2 (an ideal case): True and predicted $m(x,t)$ over time for 
$u(x,t) = 0.1\cos{(2\protect\pi x)}(t+1)$ and $m(x,0)=0.5$.}
\label{fig:2-m}
\end{figure}

\subsection{A more realistic case}

In this case, since the true solution cannot be easily found, we will only
show the predicted solution $u(x,t)$, $m(x,t)$ and the relative cost (\ref%
{num:rel-cost}) at the solution. This is a realistic scenario when the
initial data $u(x,0)$ and $m(x,0)$ are given, but the solution of the MFG
system is not known. Furthermore, it is also unknown whether or not the
chosen initial data are traces of any solution of the MFG system. The goal
is to predict the solution $u(x,t)$ and $m(x,t)$ using the convexification
method and evaluate the quality of the prediction by the relative cost (\ref%
{num:rel-cost}). Also, for the sake of brevity, we will not show the graphs
of the first order optimality.

\textbf{Test 2.1:} We choose $u(x,0)=(x^{2}-1)^{2}$ and $m(x,0)$ a Gaussian
defined as 
\begin{equation*}
m(x,0)=%
\begin{cases}
5.57\ \text{exp}\left( -{0.4^{2}}/(x^{2}-0.4^{2})\right) , & \text{if }%
|x|<0.4, \\ 
0, & \text{otherwise}.%
\end{cases}%
\end{equation*}%
The constant $5.57$ is the normalization constant. The predicted $u(x,t)$
and $m(x,t)$ are shown in Figures \ref{fig:C21-u} and \ref{fig:C21-m},
respectively. The relative cost at the solution is shown in Figure \ref%
{fig:costs} (a).

\begin{figure}[ht!]
\centering
\subfloat[$t=0$]{\includegraphics[width=0.3\textwidth]{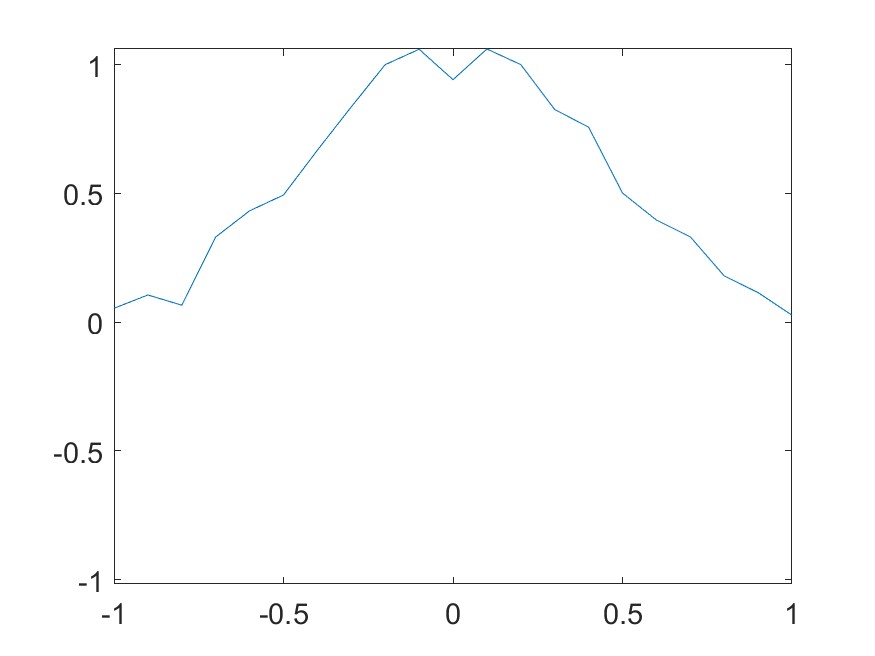}} %
\subfloat[$t=0.6$]{\includegraphics[width=0.3\textwidth]{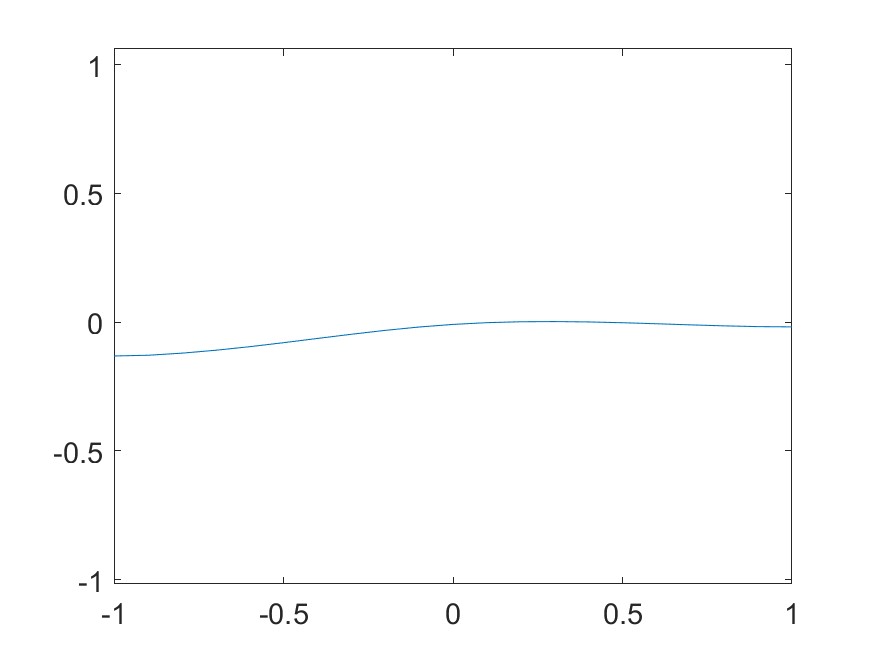}} %
\subfloat[$t=1$]{\includegraphics[width=0.3\textwidth]{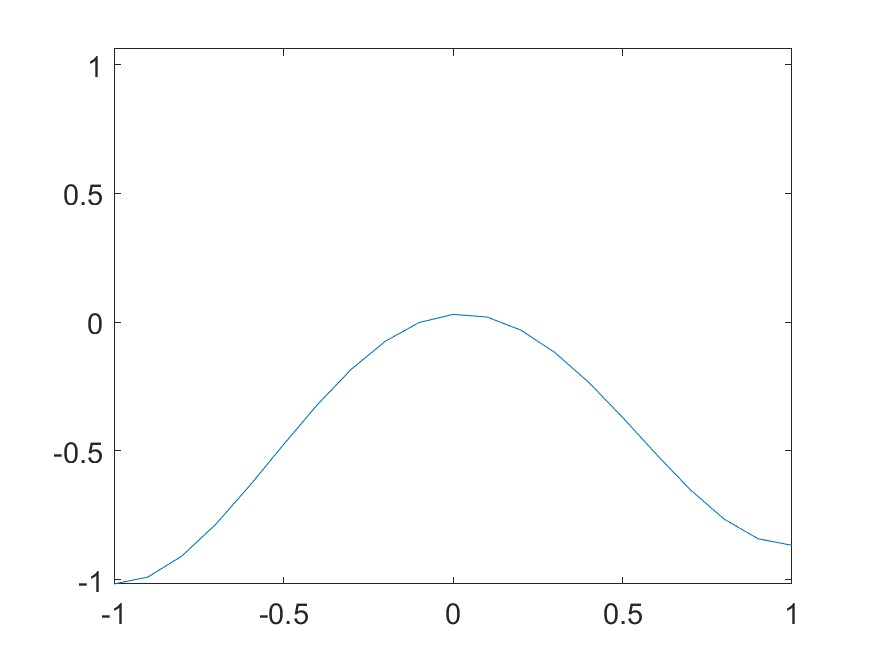}}
\caption{Test 2.1 (a more realistic case): Predicted $u(x,t)$ over time for $%
u(x,0) = (x^2-1)^2$ and $m(x,0)$ a Gaussian.}
\label{fig:C21-u}
\end{figure}

\begin{figure}[ht!]
\centering
\subfloat[$t=0$]{\includegraphics[width=0.3\textwidth]{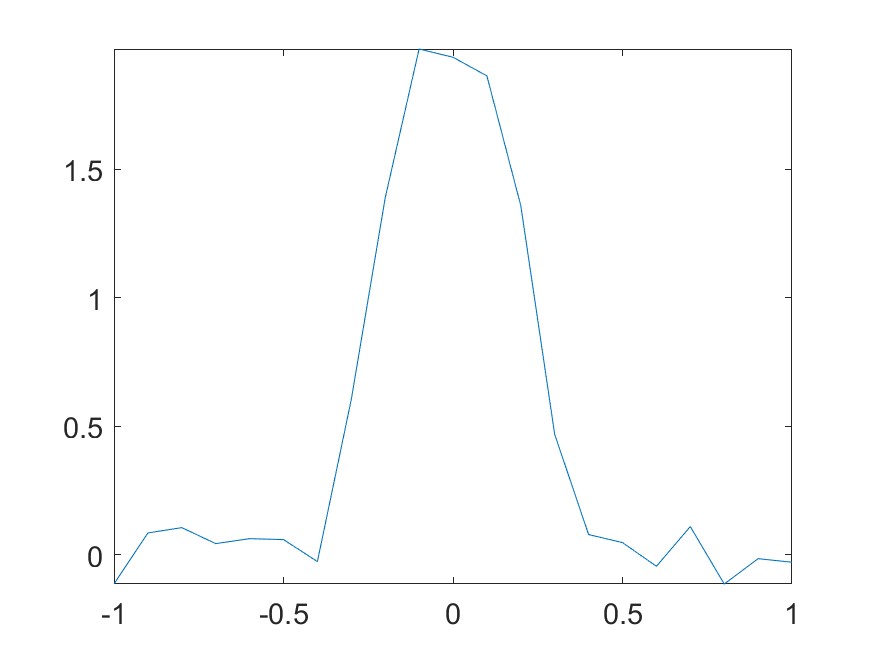}} %
\subfloat[$t=0.6$]{\includegraphics[width=0.3\textwidth]{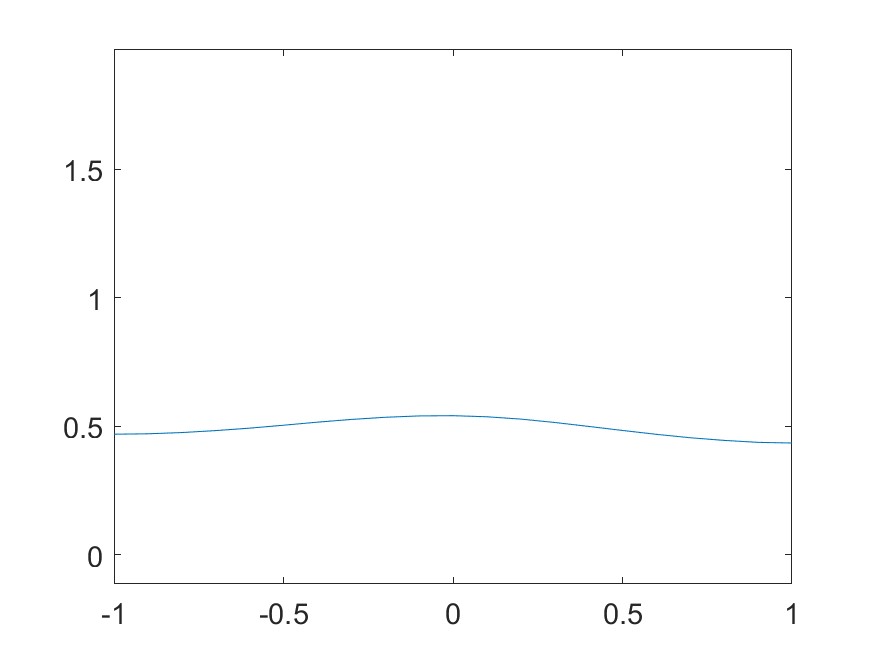}} %
\subfloat[$t=1$]{\includegraphics[width=0.3\textwidth]{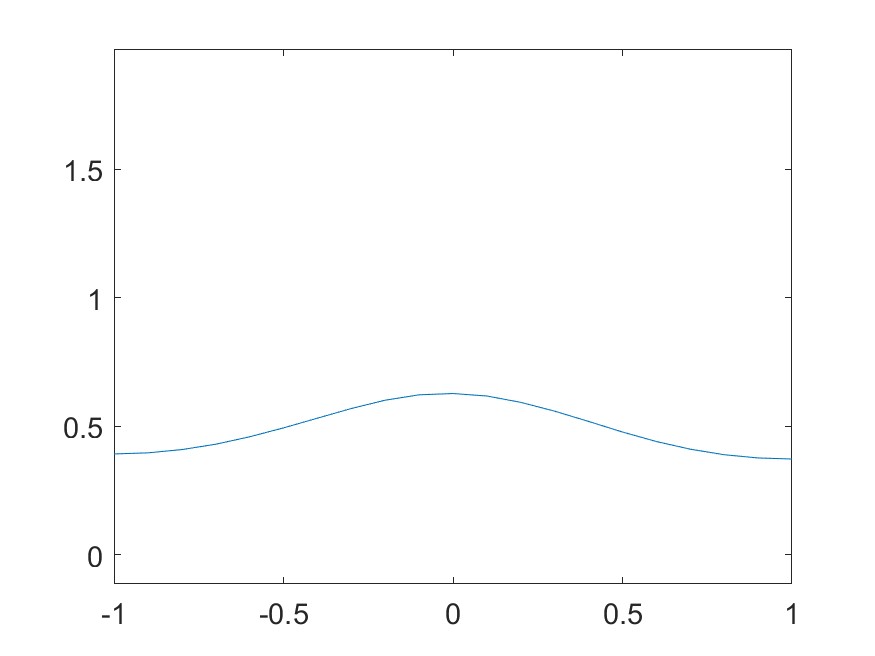}}
\caption{Test 2.1 (a more realistic case): Predicted $m(x,t)$ over time for $%
u(x,0) = (x^2-1)^2$ and $m(x,0)$ a Gaussian.}
\label{fig:C21-m}
\end{figure}

\textbf{Test 2.2:} We choose $u(x,0)=\cos {(2\pi x)}$ and $m(x,0)=0.5$. The
predicted $u(x,t)$ and $m(x,t)$ are shown in Figures \ref{fig:C22-u} and \ref%
{fig:C22-m}, respectively. The relative cost at the solution is shown in
Figure \ref{fig:costs} (b).

\begin{figure}[ht!]
\centering
\subfloat[$t=0$]{\includegraphics[width=0.3\textwidth]{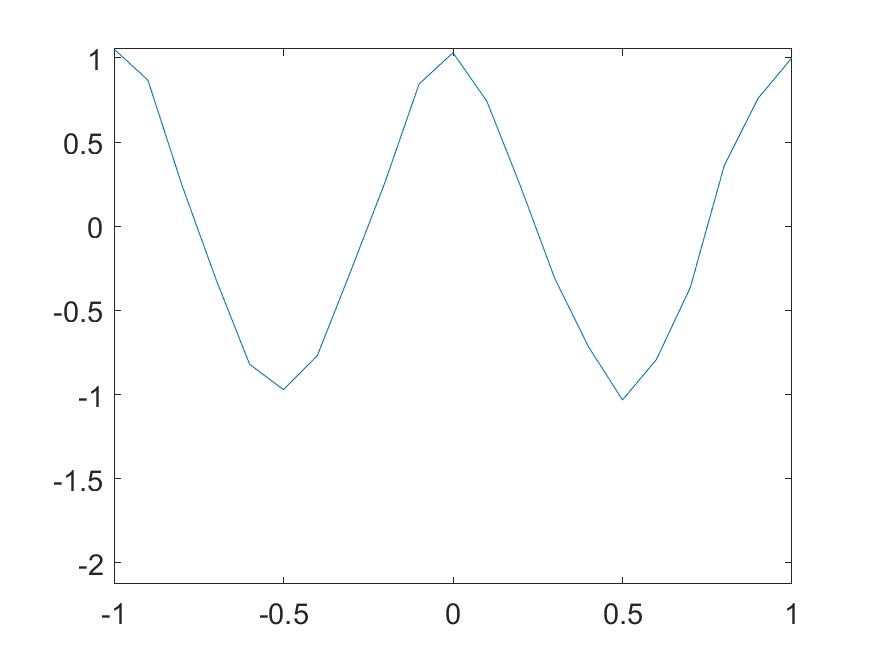}} %
\subfloat[$t=0.6$]{\includegraphics[width=0.3\textwidth]{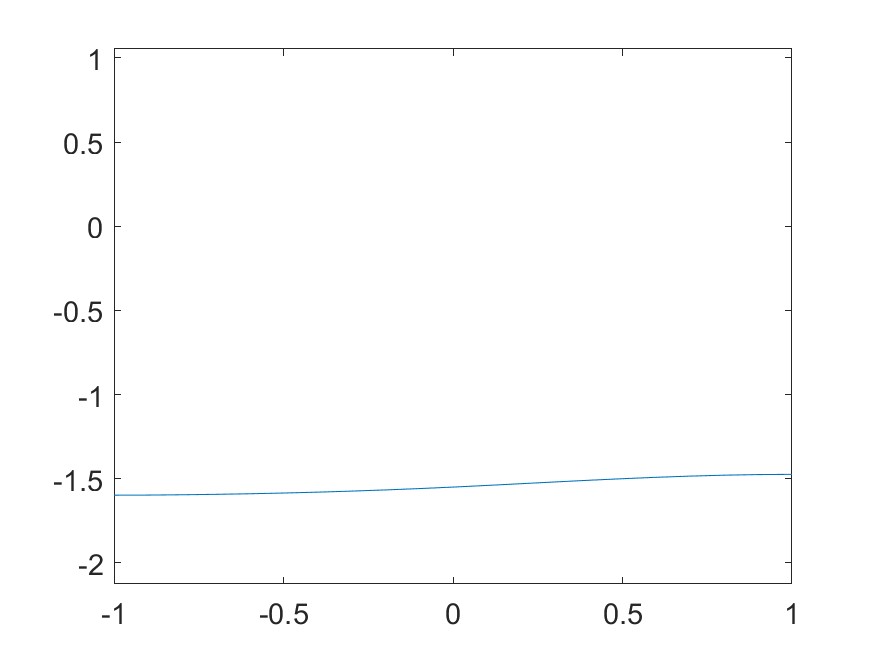}} %
\subfloat[$t=1$]{\includegraphics[width=0.3\textwidth]{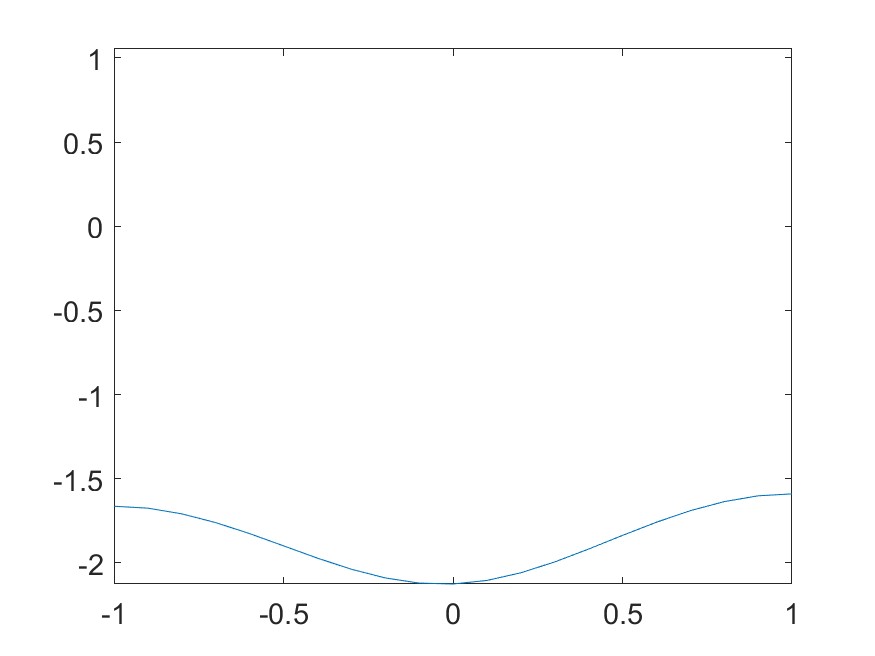}}
\caption{Test 2.2 (a more realistic case): Predicted $u(x,t)$ over time for $%
u(x,0) = \cos{(2\protect\pi x)}$ and $m(x,0)=0.5$.}
\label{fig:C22-u}
\end{figure}

\begin{figure}[ht!]
\centering
\subfloat[$t=0$]{\includegraphics[width=0.3\textwidth]{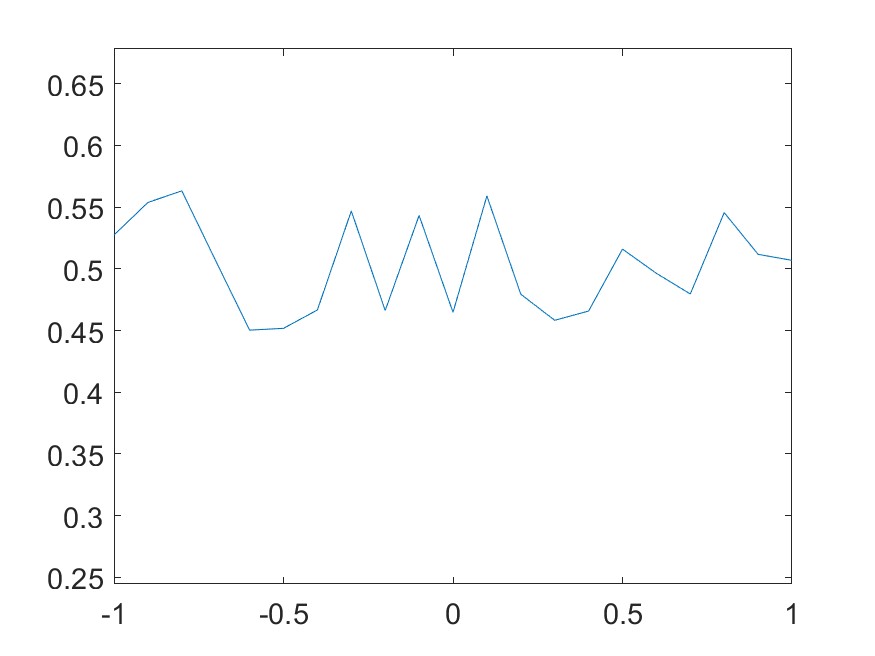}} %
\subfloat[$t=0.6$]{\includegraphics[width=0.3\textwidth]{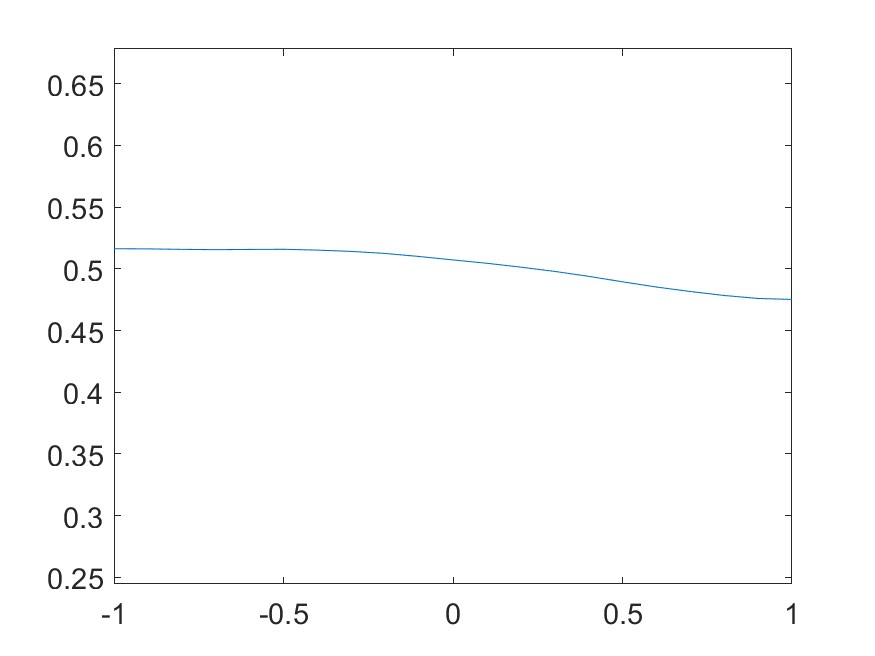}} %
\subfloat[$t=1$]{\includegraphics[width=0.3\textwidth]{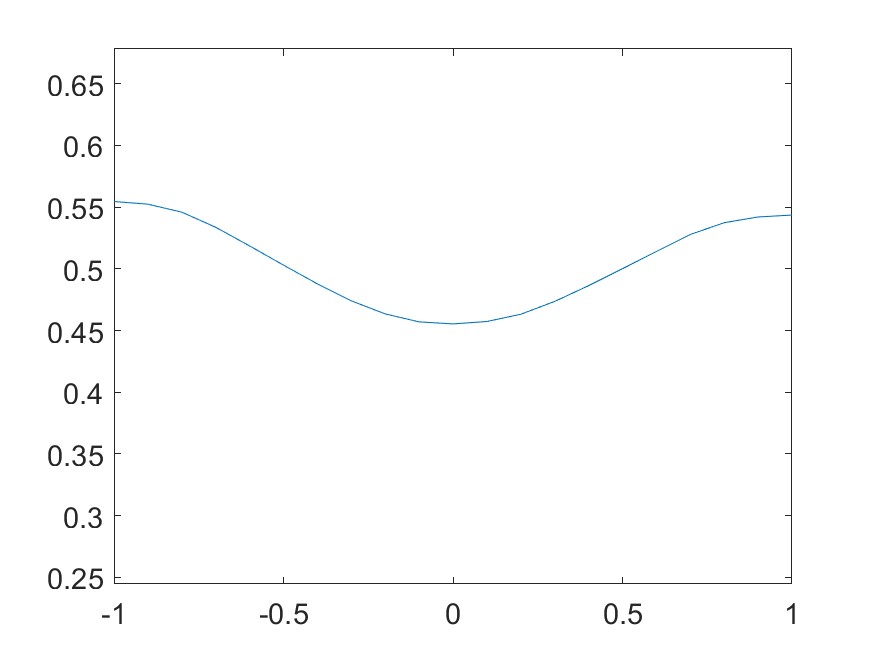}}
\caption{Test 2.2 (a more realistic case): Predicted $m(x,t)$ over time for $%
u(x,0) = \cos{(2\protect\pi x)}$ and $m(x,0)=0.5$.}
\label{fig:C22-m}
\end{figure}

\subsection{The Case Mimicking the Real-Life Data}

In this test, we choose the initial data $u(x,0)$ and $m(x,0)$ to be
somewhat close to reality. The goal is to estimate how well the
convexification method can predict the solution of the MFG system in a
real-life scenario. In fact, we can computationally simulate the data, which
are close to the real ones for the public sentiments. The topic of
applications of the technique developed in this paper is the subject of our
future publication and is, therefore, outside of the scope of this paper.

\textbf{Test 3.1:} We choose $u(x,0)$ as a smooth transition from $-0.5$ to $%
0.5$ defined as 
\begin{equation*}
u(x,0)=\frac{\tau ((1+x)/2)}{\tau ((1+x)/2)+\tau ((1-x)/2)}-0.5,\ \text{%
where }\tau (x)=%
\begin{cases}
\text{exp}(-1/x^{2}), & \text{if }x>0, \\ 
0, & \text{if }x\leq 0.%
\end{cases}%
.
\end{equation*}%
and $m(x,0)$ as a decentered Gaussian defined as 
\begin{equation*}
m(x,0)=%
\begin{cases}
5.57\ \text{exp}\left( -\frac{0.4^{2}}{(x-0.5)^{2}-0.4^{2}}\right) , & \text{%
if }|x-0.5|<0.4, \\ 
0, & \text{otherwise}.%
\end{cases}%
\end{equation*}%
The predicted $u(x,t)$ and $m(x,t)$ are shown in Figures \ref{fig:C23-u} and %
\ref{fig:C23-m}, respectively. The relative cost at the solution is shown in
Figure \ref{fig:costs} (c).

\begin{figure}[ht!]
\centering
\subfloat[$t=0$]{\includegraphics[width=0.3\textwidth]{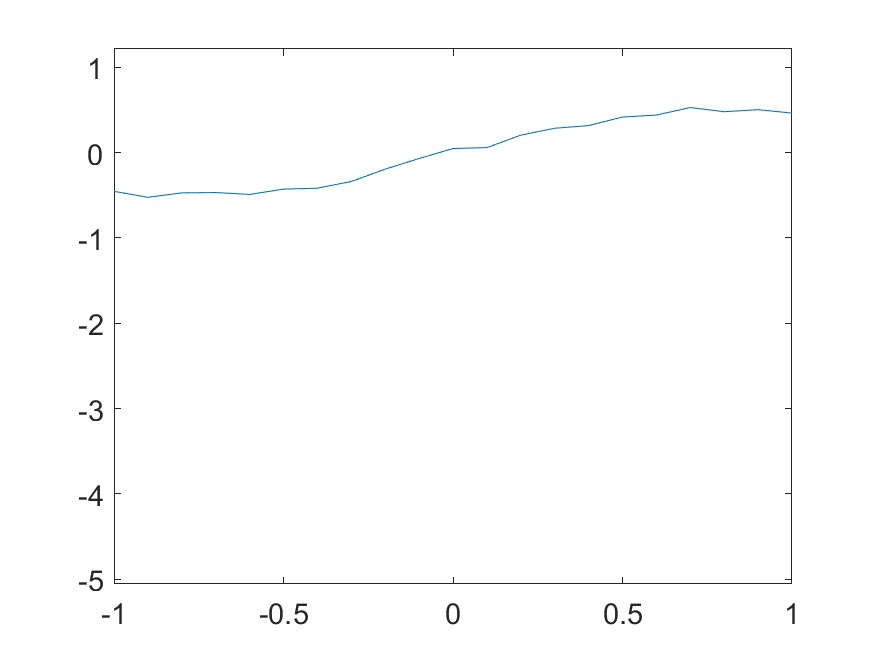}} %
\subfloat[$t=0.6$]{\includegraphics[width=0.3\textwidth]{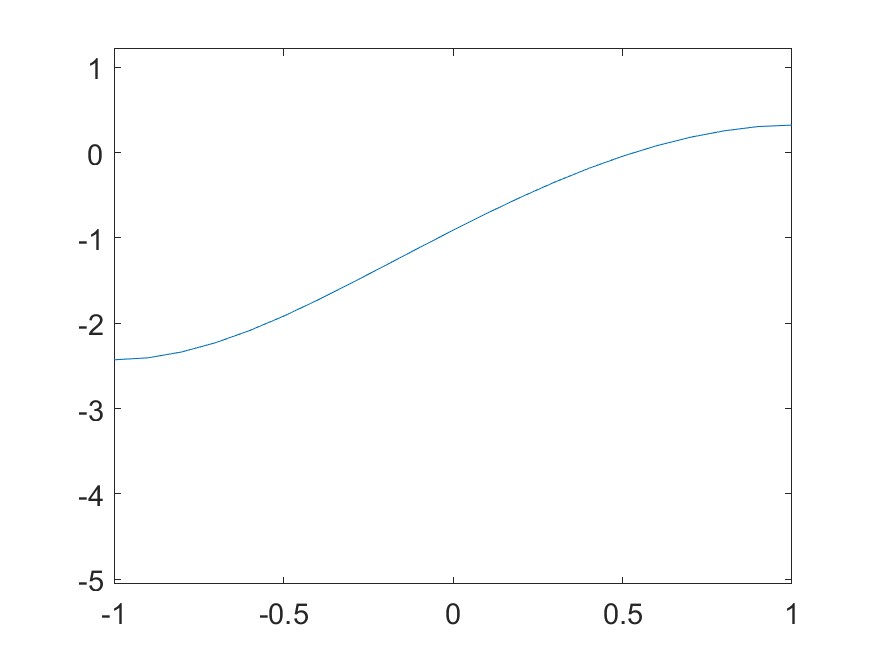}} %
\subfloat[$t=1$]{\includegraphics[width=0.3\textwidth]{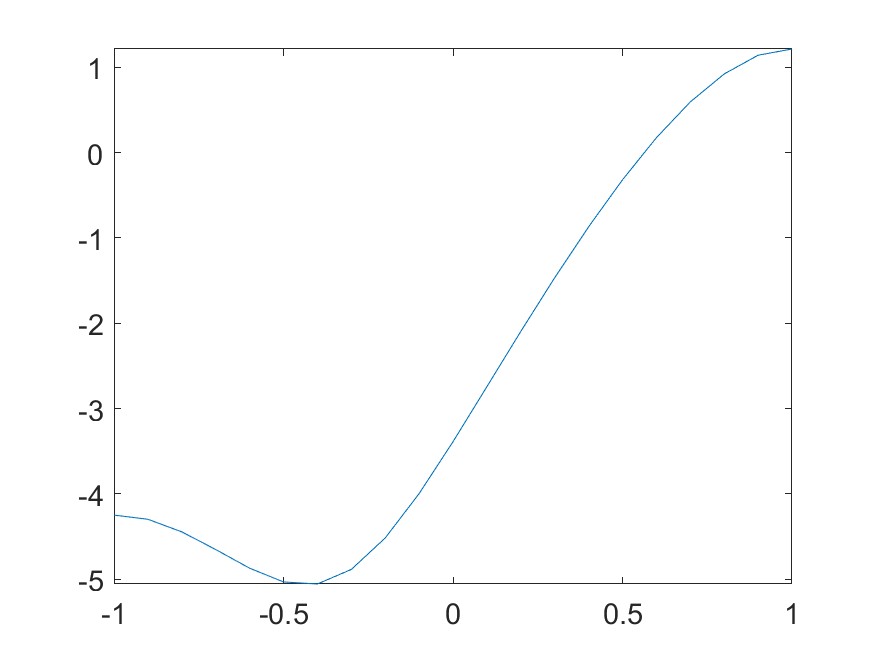}}
\caption{Test 3.1 (mimicking the reality): Predicted $u(x,t)$ over time for $%
u(x,0)$ as a smooth transition from $-0.5$ to $0.5$ and $m(x,0)$ as a
decentered Gaussian.}
\label{fig:C23-u}
\end{figure}

\begin{figure}[ht!]
\centering
\subfloat[$t=0$]{\includegraphics[width=0.3\textwidth]{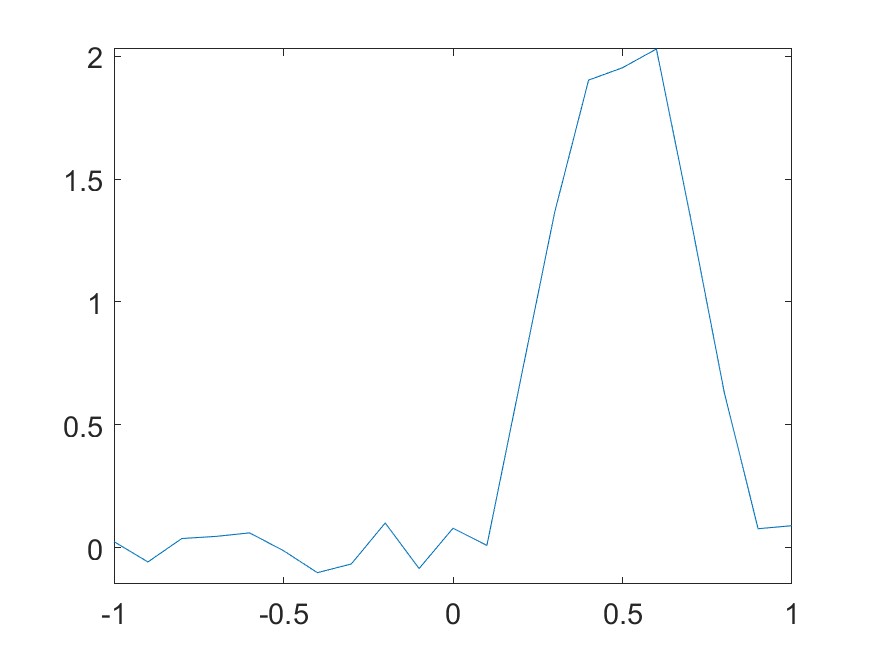}} %
\subfloat[$t=0.6$]{\includegraphics[width=0.3\textwidth]{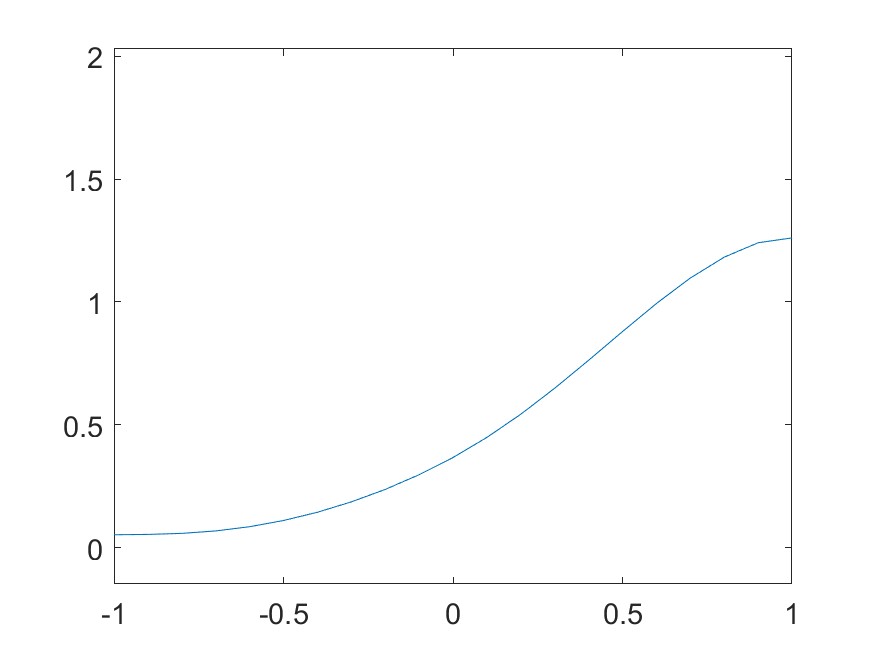}} %
\subfloat[$t=1$]{\includegraphics[width=0.3\textwidth]{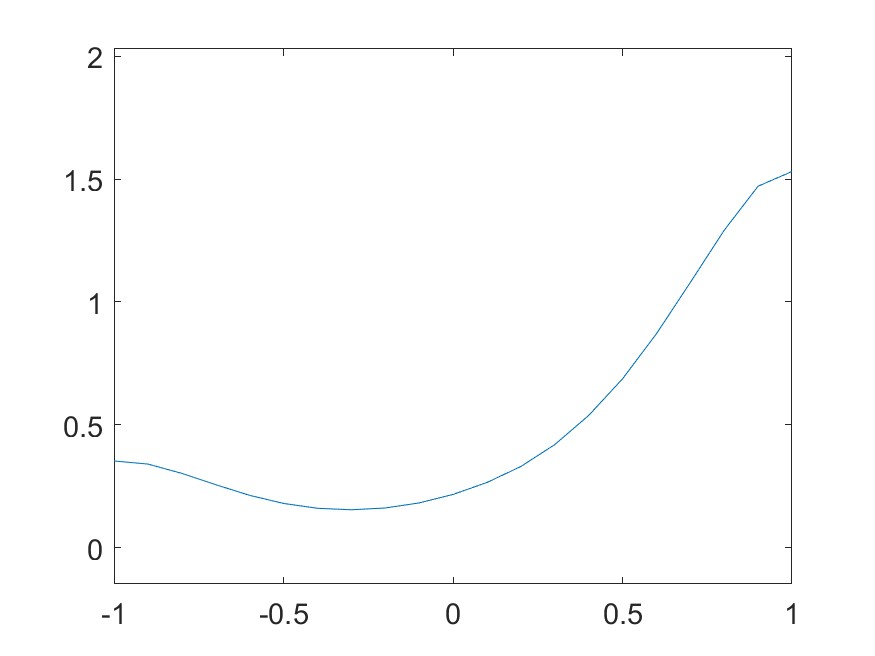}}
\caption{Test 3.1 (mimicking the reality): Predicted $m(x,t)$ over time for $%
u(x,0)$ as a smooth transition from $-0.5$ to $0.5$ and $m(x,0)$ as a
decentered Gaussian.}
\label{fig:C23-m}
\end{figure}

\begin{figure}[ht!]
\centering
\subfloat[Test 2.1]{\includegraphics[width=0.3\textwidth]{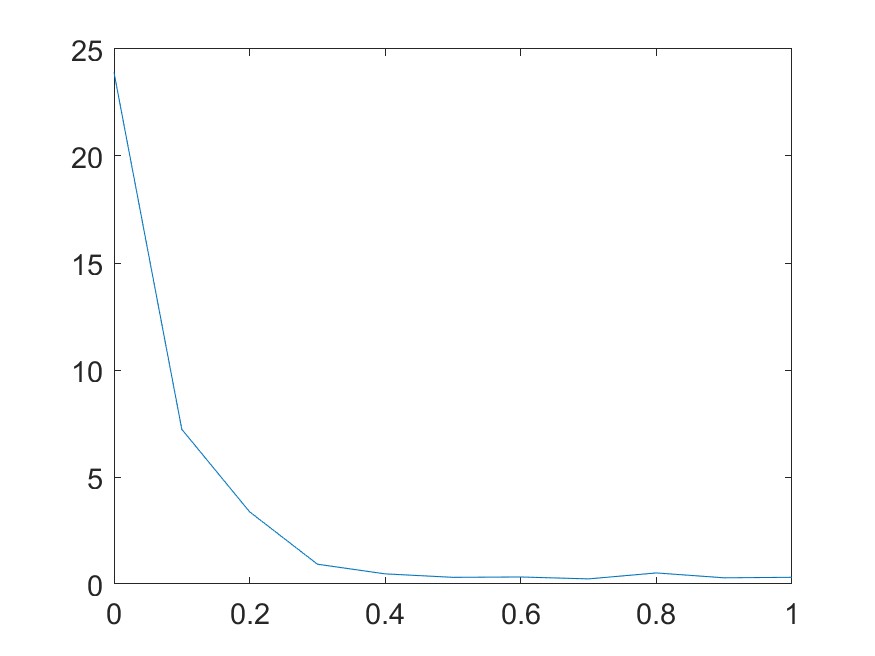}} %
\subfloat[Test 2.2]{\includegraphics[width=0.3\textwidth]{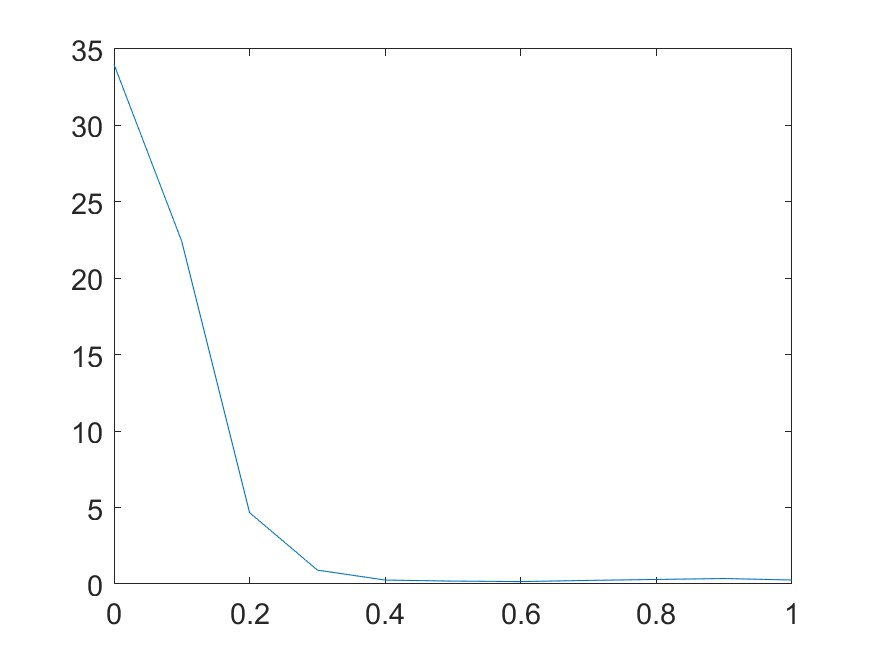}} %
\subfloat[Test 3.1]{\includegraphics[width=0.3\textwidth]{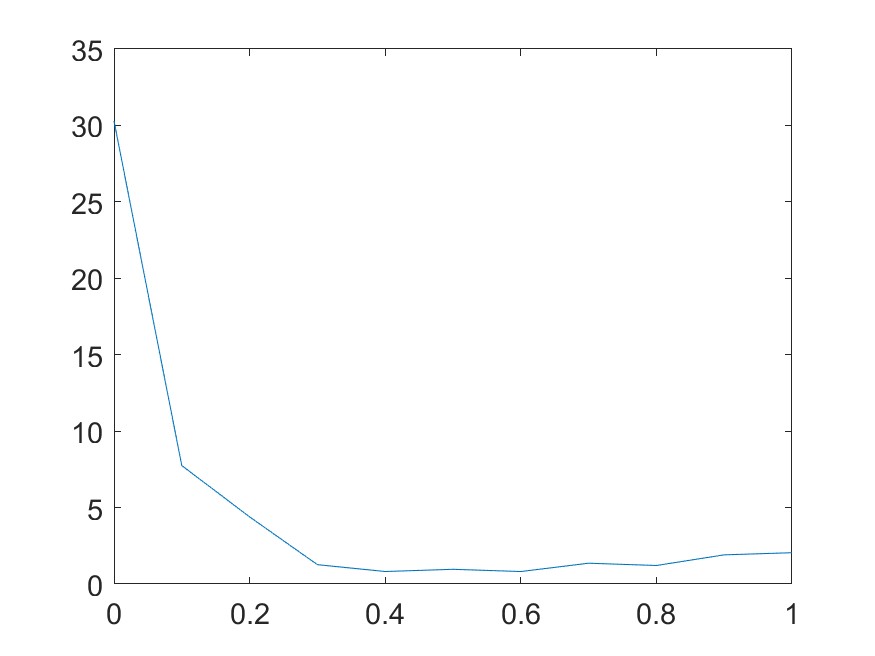}}
\caption{Relative cost (\protect\ref{num:rel-cost}) at the solution for
Tests (a) 2.1, (b) 2.2, and (c) 3.1.}
\label{fig:costs}
\end{figure}

\subsection{An Interesting Observation}

An interesting phenomenon, which one can observe from (\ref{num:rel-cost})
and Figures \ref{fig:1-cost-extended}, \ref{fig:costs} is that for small
times $t\in \left( 0,0.3\right) $, the minimizer of our target functional (%
\ref{5.6}) does not satisfy well to the MFG system (\ref{num:1})-(\ref{num:2}%
) for all considered cases, even for the ideal initial conditions. On the
other hand, it satisfies that system well for $t\in \left[ 0.3,1\right) $.
Finally, the instability for the function $u(x,t)$ takes over for $t>1$, see
Figure \ref{fig:1-cost-extended}.

\section*{Acknowledgment}

This research effort was supported by the National Science Foundation grant
DMS 2436227.

\bibliographystyle{siamplain}
\bibliography{references}

\begin{thebibliography}{10}

\bibitem{A}
{\sc Y.~Achdou, P.~Cardaliaguet, F.~Delarue, A.~Porretta, and F.~Santambrogio}, {\em Mean Field Games}, volume 2281 of Lecture Notes in Mathematics, C. I. M. E. Foundation Subseries, Springer, Nature, Cetraro, Italy, 2019, \url{https://doi.org/10.1007/978-3-030-59837-2}.

\bibitem{Bak}
{\sc A.~B. Bakushinskii, M.~V. Klibanov, and N.~A. Koshev}, {\em Carleman weight functions for a globally convergent numerical method for ill-posed cauchy problems for some quasilinear pdes}, Nonlinear Analysis: Real World Applications, 34 (2017), pp.~201--224, \url{https://doi.org/10.1016/j.nonrwa.2016.08.008}.

\bibitem{Banez}
{\sc R.~A. Banez, H.~Gao, L.~Li, C.~Yang, Z.~Han, and H.~V. Poor}, {\em Belief and opinion evolution in social networks based on a multi-population mean field game approach}, in 2020 IEEE International Conference on Communications, 2020, pp.~1--6, \url{https://doi.org/10.1109/ICC40277.2020.9148985}.

\bibitem{Bauso}
{\sc D.~Bauso, H.~Tembine, and T.~Basar}, {\em Opinion dynamics in social networks through mean-field games}, SIAM J. Control Optim., 54 (2016), pp.~3225--3257, \url{https://doi.org/10.1137/140985676}.

\bibitem{Chamley}
{\sc C.~Chamley, A.~Scaglione, and L.~Li}, {\em Models for the diffusion of beliefs in social networks: An overview}, IEEE Signal Processing Magazine, 30 (2013), pp.~16--29, \url{https://doi.org/10.1109/MSP.2012.2234508}.

\bibitem{Chofri}
{\sc S.~E. Chorfi, A.~Habbal, M.~Jahid, L.~Maniar, and A.~Ratnani}, {\em Stability of backward inverse problems for degenerate mean-field game systems}, Evolution Equations and Control Theory, 14 (2025), pp.~1257--1280, \url{https://doi.org/10.3934/eect.2025033}.

\bibitem{Chow}
{\sc Y.~T. Chow, S.~W. Fung, S.~Liu, L.~Nurbekyan, and S.~Osher}, {\em A numerical algorithm for inverse problem from partial boundary measurement arising from mean field game problem}, Inverse Problems, 39 (2023), p.~014001, \url{https://doi.org/10.1088/1361-6420/aca5b0}.

\bibitem{Festa}
{\sc A.~Festa, S.~G\"{o}ttlich, and M.~Ricciardi}, {\em Forward-forward mean field games in mathematical modeling with application to opinion formation and voting models}, Dyn. Games Appl., 15 (2022), p.~2025, \url{https://doi.org/10.1007/s13235-024-00578-3}.

\bibitem{Gao}
{\sc H.~Gao, A.~Lin, R.~A. Banez, W.~Li, Z.~Han, S.~Osher, and H.~V. Poor}, {\em Opinion evolution in social networks: Connecting mean field games to generative adversarial nets}, IEEE Trans. Network Sci. Eng., 9 (2022), p.~2022, \url{https://doi.org/10.1109/TNSE.2022.3169057}.

\bibitem{Huang1}
{\sc M.~Huang, P.~E. Caines, and R.~P. Malham\'{e}}, {\em Large-population cost-coupled lqg problems with nonuniform agents: individual-mass behavior and decentralized {N}ash equilibria}, IEEE Trans. Automat. Control, 52 (2007), pp.~1560--1571, \url{https://doi.org/10.1109/TAC.2007.904450}.

\bibitem{Huang2}
{\sc M.~Huang, R.~P. Malham\'{e}, and P.~E. Caines}, {\em Large population stochastic dynamic games: closed-loop mckean-vlasov systems and the {N}ash certainty equivalence principle}, Commun. Inf. Syst., 6 (2006), pp.~221--251, \url{https://doi.org/10.4310/CIS.2006.v6.n3.a5}.

\bibitem{Isakov}
{\sc V.~Isakov}, {\em Inverse Problems for Partial Differential Equations}, Springer, New York, 2006.

\bibitem{Klib97}
{\sc M.~V. Klibanov}, {\em Global convexity in a three-dimensional inverse acoustic problem}, SIAM Journal on Mathematical Analysis, 28 (1997), pp.~1371--1388, \url{https://doi.org/10.1137/S0036141096297364}.

\bibitem{MFG2}
{\sc M.~V. Klibanov}, {\em The mean-field games system: Carleman estimates, lipschitz stability and uniqueness}, J. Inverse Ill-Posed Probl., 31 (2023), pp.~455--466, \url{https://doi.org/10.1515/jiip-2023-0023}.

\bibitem{MFG1}
{\sc M.~V. Klibanov and Y.~Averboukh}, {\em Lipschitz stability estimate and uniqueness in the retrospective analysis for the mean field games system via two carleman estimates}, SIAM J. Mathematical Analysis, 56 (2024), pp.~616--636, \url{https://doi.org/10.1137/23M1554801}.

\bibitem{Klib95}
{\sc M.~V. Klibanov and O.~V. Ioussoupova}, {\em Uniform strict convexity of a cost functional for three-dimensional inverse scattering problem}, SIAM Journal on Mathematical Analysis, 26 (1995), pp.~147--179, \url{https://doi.org/10.1137/S0036141093244039}.

\bibitem{SAR}
{\sc M.~V. Klibanov, V.~A. Khoa, A.~V. Smirnov, L.~H. Nguyen, G.~W. Bidney, L.~Nguyen, A.~Sullivan, and V.~N. Astratov}, {\em Convexification inversion method for nonlinear {SAR} imaging with experimentally collected data}, Journal of Applied and Industrial Mathematics, 15 (2021), pp.~413--436, \url{https://doi.org/10.1134/S1990478921030054}.

\bibitem{KL}
{\sc M.~V. Klibanov and J.~Li}, {\em Inverse Problems and Carleman Estimates: Global Uniqueness, Global Convergence and Experimental Data}, De Gruyter, Berlin, 2021, \url{https://doi.org/10.1515/9783110745481}.

\bibitem{MFGbook}
{\sc M.~V. Klibanov and J.~Li}, {\em Carleman Estimates in Mean Field Games}, De Gruyter, 2025.

\bibitem{MFG4}
{\sc M.~V. Klibanov, J.~Li, and H.~Liu}, {\em H\"{o}lder stability and uniqueness for the mean field games system via carleman estimates}, Studies in Applied Mathematics, 151 (2023), pp.~1447--1470, \url{https://doi.org/10.1111/sapm.12633}.

\bibitem{MFG7}
{\sc M.~V. Klibanov, J.~Li, and Z.~Yang}, {\em Convexification numerical method for the retrospective problem of mean field games}, Applied Mathematics and Optimization, 90 (2024), \url{https://doi.org/10.1007/s00245-024-10152-3}.

\bibitem{MFGCAMWA}
{\sc M.~V. Klibanov, J.~Li, and Z.~Yang}, {\em Convexification for a coefficient inverse problem for a system of two coupled nonlinear parabolic equations}, Computers and Mathematics with Applications, 179 (2025), pp.~41--58, \url{https://doi.org/10.1016/j.camwa.2024.12.004}.

\bibitem{Kexper}
{\sc M.~V. Klibanov, K.~McGoff, T.~Truong, W.~Xin, S.~Yin, and S.~Chen}, {\em Toward practical forecasts of public sentiments via convexification for mean field games: evidence from covid-19 real data}.
\newblock in preparation.

\bibitem{LL1}
{\sc J.-M. Lasry and P.-L. Lions}, {\em Jeux \`{a} champ moyen. i. le cas stationnaire}, C. R. Math. Acad. Sci. Paris, 343 (2006), pp.~619--625, \url{https://doi.org/10.1016/j.crma.2006.09.019}.

\bibitem{LL2}
{\sc J.-M. Lasry and P.-L. Lions}, {\em Mean field games}, Japanese Journal of Mathematics, 2 (2007), pp.~229--260, \url{https://doi.org/10.1007/s11537-007-0657-8}.

\bibitem{Liu2}
{\sc H.~Liu, C.~W.~K. Lo, and S.~Zhang}, {\em Decoding a mean field game by the cauchy data around its unknown stationary states}, J. Lond. Math. Soc., 111 (2025), p.~e70173, \url{https://doi.org/10.1112/jlms.70173}.

\bibitem{Liu1}
{\sc H.~Liu, C.~Mou, and S.~Zhang}, {\em Inverse problems for mean field games}, Inverse Problems, 39 (2023), p.~085003, \url{https://doi.org/10.1088/1361-6420/acdd90}.

\bibitem{Mossel}
{\sc E.~Mossel and O.~Tamuz}, {\em Opinion exchange dynamics}, Probability Surveys, 14 (2017), pp.~155--204, \url{https://doi.org/10.1214/14-PS230}.

\bibitem{Ren}
{\sc K.~Ren, N.~Soedjak, and K.~Wang}, {\em Unique determination of cost functions in a multipopulation mean field game model}, J. Differential Equations, 427 (2025), pp.~843--867, \url{https://doi.org/10.1016/j.jde.2025.02.037}.

\bibitem{Ren2}
{\sc K.~Ren, N.~Soedjak, K.~Wang, and H.~Zhai}, {\em Reconstructing a state-independent cost function in a mean-field game model}, Inverse Problems, 40 (2024), p.~105010, \url{https://doi.org/10.1088/1361-6420/ad7497}.

\bibitem{Stella}
{\sc L.~Stella, F.~Bagagiolo, D.~Bauso, and G.~Como}, {\em Opinion dynamics and stubbornness through mean-field games}, in 52nd IEEE Conference on Decision and Control, 2013, pp.~2519--2524, \url{https://doi.org/10.1109/CDC.2013.6760259}.

\bibitem{T}
{\sc A.~N. Tikhonov, A.~V. Goncharsky, V.~V. Stepanov, and A.~G. Yagola}, {\em Numerical Methods for the Solution of Ill-Posed Problems}, Kluwer, London, 1995.

\bibitem{Trusov}
{\sc N.~V. Trusov}, {\em Numerical study of the stock market crises based on mean field games approach}, J. Inverse Ill-Posed Probl., 29 (2021), pp.~849--865, \url{https://doi.org/10.1515/jiip-2020-0016}.

\bibitem{Yu}
{\sc J.~Yu, Q.~Xiao, T.~Chen, and R.~Lai}, {\em A bilevel optimization method for inverse mean-field games}, Inverse Problems, 40 (2024), p.~105016, \url{https://doi.org/10.1088/1361-6420/ad75b0}.

\end{thebibliography}

\end{document}